\documentclass[11pt]{article}

\usepackage{latexsym}
\usepackage{amssymb}
\usepackage{amsthm}
\usepackage{amscd}
\usepackage{amsmath}
\usepackage{mathrsfs}
\usepackage[all]{xy}
\usepackage{hyperref} 
\usepackage[usenames,dvipsnames]{color}
\usepackage{graphicx,eepic}
\usepackage{float}

\usepackage{color}

\theoremstyle{definition}
\newtheorem* {theorem*}{Theorem}

\theoremstyle{definition}
\newtheorem{theorem}{Theorem}[section]

\theoremstyle{definition}

\theoremstyle{definition}
\newtheorem{observation}{Observation}[section]

\theoremstyle{definition}

\theoremstyle{definition}

\newtheorem{lemma}{Lemma}[section]
\theoremstyle{definition}
\newtheorem{definition}{Definition}[section]
\theoremstyle{definition}

\theoremstyle{definition}

\newtheorem{proposition}{Proposition}[section]
\newtheorem{corollary}{Corollary}[section]
\newtheorem* {remark}{Remark}

\theoremstyle{definition}

\theoremstyle{definition}
\newtheorem* {remarks}{Remarks}

\xyoption{dvips}

\usepackage{fullpage}

\numberwithin{equation}{section}

\newcommand{\Mat}{\mathrm{Mat}} 
\def\({\left(}
\def\){\right)}
 \newcommand{\FF}{\mathbb{F}}  \newcommand{\CC}{\mathbb{C}}      
   
   \newcommand{\cI}{\mathcal{I}}
 
\newcommand{\cC}{\mathcal{C}}

\def\tr{\mathrm{tr}}  \def\Hom{\mathrm{Hom}}  \def\ZZ{\mathbb{Z}}   \def\H{\mathcal{H}} \def\Ind{\mathrm{Ind}}  \def\Res{\mathrm{Res}}    \def\spanning{\textnormal{-span}}   
\def\Irr{\mathrm{Irr}}  \def\wt{\widetilde}
  
\def\diag{\mathrm{diag}}   
   
   \newcommand{\fkn}{\mathfrak{n}}

\def\fk{\mathfrak}

\def\barr{\begin{array}}
\def\earr{\end{array}}
\def\ba{\begin{aligned}}
\def\ea{\end{aligned}}
\def\be{\begin{equation}}
\def\ee{\end{equation}}

\def\ol{\widehat}
\def\olfkl{\ol{\fk l}}
\def\olfks{\ol{\fk s}}
\def\oll{\ol{L}}
\def\ols{\ol{S}}

 \def\GL{\textbf{GL}} 
\def\UT{\textbf{U}}
\def\fkt{\fk{u}}

\def\Arc{\mathrm{Arc}}

\newcommand\Bell[2]{\mathcal{B}_{#1}(#2)}
\newcommand\Cat[2]{\mathcal{N}_{#1}(#2)}

\newcommand\Fe[2]{\mathcal{F}_{#1}(#2)}
\newcommand\Del[2]{\mathcal{P}_{#1}(#2)}
\newcommand\preHe[2]{\mathcal{H}_{#1}(#2)}

\newcommand\He[2]{ \wt{\mathcal{H}}_{#1}(#2)}

\newcommand\preIn[2]{\mathcal{I}_{#1}(#2)}
\newcommand\In[2]{ \wt{\mathcal{I}}_{#1}(#2)}

\newcommand\AltBell[2]{\mathcal{B}^\hom_{#1}(#2)}
\newcommand\AltCat[2]{\mathcal{N}^\hom_{#1}(#2)}
\newcommand\AltDel[2]{\mathcal{P}^\hom_{#1}(#2)}
\newcommand\AltHe[2]{\mathcal{H}^\hom_{#1}(#2)}

\def\Lattice{\mathscr{L}}

\def\NC{\mathrm{N  C}}

\def\sK{\mathscr{L}_{\mathrm{Pell}}}
\def\sL{\mathscr{L}_{\mathrm{Heis}}}
\def\sLL{\wt{\mathscr{L}}_{\mathrm{Heis}}}
\def\uu{\uparrow\hspace{-1mm}\uparrow}

\def\ds{\displaystyle}

\def\sM{\mathscr{L}_{\mathrm{Inv}}}
\def\sMM{\wt{\mathscr{L}}_{\mathrm{Inv}}}

\def\D{D}
\def\H{D'}
\def\I{D''}

\def\shift{\text{\emph{shift}}}

\makeatletter
\renewcommand{\@makefnmark}{\mbox{\textsuperscript{}}}
\makeatother

\allowdisplaybreaks[1]

\UseCrayolaColors

\begin{document}
\title{Heisenberg characters, unitriangular groups, and Fibonacci numbers}
\author{Eric Marberg
\footnote{This research was conducted with government support under the Department of Defense, Air Force Office of Scientific Research, National Defense Science and Engineering Graduate (NDSEG) Fellowship, 32 CFR 168a.} \\ Department of Mathematics \\ Massachusetts Institute of Technology \\ \tt{emarberg@math.mit.edu}
}
\date{}

\maketitle

\def\cX{\mathcal{X}}
\def\cY{\mathcal{Y}}

\def\SCh{\mathrm{SCh}}

\def\sP{\Pi}
\def\hom{\sigma}
\def\omdef{:}

\setcounter{tocdepth}{2}

\newcommand\stirl[2]
{{ #1 \atopwithdelims \{\} #2}}
\newcommand\stirlstirl[2]
{\left\{ \hspace{-1.5mm}{ #1 \atopwithdelims \{\} #2} \hspace{-1.5mm}\right\}}

\newcommand\stirlb[2]
{{ #1 \atopwithdelims \{\} #2}_{\hspace{-0.5mm}\mathrm{B}}\hspace{0.5mm}}

\newcommand\stirlstirlX[2]
{\left\{ \hspace{-1.5mm}{ #1 \atopwithdelims \{\} #2} \hspace{-1.5mm}\right\}_{\hspace{-0.5mm}\mathrm{X}}\hspace{0.5mm}}
\newcommand\stirlstirlB[2]
{\left\{ \hspace{-1.5mm}{ #1 \atopwithdelims \{\} #2} \hspace{-1.5mm}\right\}_{\hspace{-0.5mm}\mathrm{B}}\hspace{0.5mm}}
\newcommand\stirlstirlD[2]
{\left\{ \hspace{-1.5mm}{ #1 \atopwithdelims \{\} #2} \hspace{-1.5mm}\right\}_{\hspace{-0.5mm}\mathrm{D}}\hspace{0.5mm}}

\newcommand\stst[2]
{\bigl\{ \hspace{-1.5mm}{ #1 \atopwithdelims \{\} #2} \hspace{-1.5mm}\bigr\}}

\newcommand\ststB[2]
{\bigl\{ \hspace{-1.5mm}{ #1 \atopwithdelims \{\} #2} \hspace{-1.5mm}\bigr\}_{\hspace{-0.5mm}\mathrm{B}}\hspace{0.5mm}}

\def\u{\uparrow}
\def\d{\nearrow}
\def\r{\to}

\def\Nat{\mathbb{N}}

\begin{abstract}
 Let $\UT_n(\FF_q)$ denote the group of unipotent $n\times n$ upper triangular matrices over a finite field with $q$ elements. 
We show that the Heisenberg characters of $\UT_{n+1}(\FF_q)$ are indexed by lattice paths from the origin to the line $x+y=n$ using the steps $(1,0), (1,1), (0,1), (0,2)$,  which are labeled in a certain way by nonzero elements of $\FF_q$. In particular, we prove for $n\geq 1$ that the number of Heisenberg characters of $\UT_{n+1}(\FF_q)$ is a polynomial in $q-1$  with nonnegative integer coefficients and degree $n$, whose leading coefficient is the $n$th Fibonacci number.  Similarly, we find that the number of Heisenberg supercharacters of $\UT_n(\FF_q)$ is a polynomial in $q-1$ whose coefficients are Delannoy numbers and whose values give a $q$-analogue for the Pell numbers. By counting the fixed points of the action of a certain group of linear characters, we prove that the numbers of supercharacters, irreducible supercharacters, Heisenberg supercharacters, and Heisenberg characters of the subgroup 
 of $\UT_n(\FF_q)$ consisting of matrices whose superdiagonal entries sum to zero are likewise all polynomials in $q-1$ with nonnegative integer coefficients.
\end{abstract}
%
%
%
\section{Introduction}

Let $\UT_n(\FF_q)$ denote the  \emph{unitriangular group}  of $n\times n$ upper triangular matrices with ones on the diagonal, over a finite field $\FF_q$ with $q$ elements.  This is a finite group with $q^{n(n-1)/2}$ elements and, writing $p$ for the prime characteristic of $\FF_q$,  a Sylow $p$-subgroup of  the general linear group $\GL_n(\FF_q)$.  Classifying the irreducible representations of $\UT_n(\FF_q)$ is a well-known wild problem (cf. the discussion in Section 2 of \cite{ADS}).  When $n<2p$, there is at least a method of constructing the irreducible representations from certain coadjoint orbits (see \cite{Sangroniz}), but even in this situation we have no way of enumerating the orbits in any accessible fashion.

Despite this, the representation theory of $\UT_n(\FF_q)$ engenders a surprising abundance of combinatorial structures--provided we refocus our attention on  irreducible representations to other  families.  
A prominent success story of this philosophy  comes from the theory of \emph{supercharacters}
introduced by Andr\'e \cite{Andre,Andre1}, simplified by Yan \cite{Yan}, and generalized by Diaconis and Isaacs \cite{DI}.
A brief overview of this subject goes as follows.

Let $\fkn$ denote a nilpotent finite-dimensional associative $\FF_q$-algebra, and write $G $ for the corresponding \emph{algebra group} of formal sums $1+X$ for $X \in \fkn$.  For example, one could take $\fkn$ to be the set $\fkt_n(\FF_q)$ of $n\times n$ strictly upper triangular matrices over $\FF_q$ so that $G = \UT_n(\FF_q)$.
Choose a nontrivial homomorphism $\theta : \FF_q^+\to \CC^\times$.
The elements of the dual space $\fkn^*$ of $\FF_q$-linear maps $\fkn\to \FF_q$ then index a basis for the group algebra $\CC G$, given by the elements
\[v_{\lambda}\omdef= \sum_{g\in G} \theta\circ\lambda(g-1)g\in \CC G,\qquad\text{for }\lambda \in \fkn^*.\] The \emph{supercharacters} of $G$ are defined as the characters of the left $G$-modules $\CC G v_\lambda$ for $\lambda \in \fkn^*$.   These characters are often reducible, but their irreducible constituents partition the set of all complex irreducible characters of the group $G$. 
In view of this and other nice properties, the set of supercharacters serves as a useful approximation to $\Irr(G)$.  In particular, the supercharacters of an algebra group form an important nontrivial example of what Diaconis and Isaacs call a \emph{supercharacter theory} of a finite group \cite{DI}.


In contrast to the irreducible characters, the supercharacters of the untriangular group may be completely classified; viz.,  by a $q$-analogue of the set partitions of $[n]\omdef = \{ i \in \ZZ : 1\leq i \leq n\}$.
    Write $\Lambda \vdash[n]$ and say that $\Lambda$ is a \emph{partition} of $[n]$ if $\Lambda$ is a set partition the union of whose blocks is $[n]$.   
A pair $(i,j) \in [n] \times [n]$ is an \emph{arc} of $\Lambda \vdash [n]$ if $i$ and $j$ occur  in the same block of $\Lambda$ such that 
$j$ is the least element of the block greater than $i$. Let $\Arc(\Lambda)$ denote the set of arcs of $\Lambda$, so that for example
$\Lambda= \{ \{ 1,3\} ,\{2,4,6,7\}, \{5\} \}\vdash[7]$ has $\Arc(\Lambda) = \{ (1,3), (2,4), (4,6), (6,7)\}$.
A set partition $\Lambda\vdash[n]$ is then \emph{noncrossing} if no two arcs $(i,k),(j,l) \in \Arc(\Lambda)$ have $i<j<k<l$.  
The $q$-analogue indexing the supercharacters of $\UT_n(\FF_q)$ is now defined thus:

\begin{definition}
An \emph{$\FF_q$-labeled partition} of $[n]$ is a set partition $\Lambda \vdash[n]$ with a map $\Arc(\Lambda)\to \FF_q^\times$.  We denote the image of $(i,j) \in \Arc(\Lambda)$ under this map by $\Lambda_{ij}$.
\end{definition}

Let $\sP(n,\FF_q)$ and $\NC(n,\FF_q)$ denote the sets of ordinary and noncrossing $\FF_q$-labeled set partitions of $[n]$. 
We may view $\sP(n,\FF_q)$ as a subset of $\fkt_n(\FF_q)^*$ by defining \be\label{ide} \Lambda (X) \omdef =  \sum_{(i,j)\in\Arc(\Lambda)} \Lambda_{ij} X_{ij},\qquad\text{for $\Lambda \in \sP(n,\FF_q)$ and $X \in \fkt_n(\FF_q)$}.\ee
Here $X_{ij}$ denotes the entry located in the $i^{\mathrm{th}}$ row and $j^{\mathrm{th}}$ column of the matrix $X$.
Write $\chi_\Lambda$ for the supercharacter of $\UT_n(\FF_q)$ corresponding to the module generated by $v_\Lambda \in \CC \UT_n(\FF_q)$.
The following remarkable result derives from the work of Andr\'e \cite{Andre,Andre1} and Yan \cite{Yan}:
the correspondence $\Lambda \mapsto \chi_\Lambda$ is a bijection 
\be\label{bell-thm} \ba \sP(n,\FF_q) &\to
  \{ \text{Supercharacters of $\UT_n(\FF_q)$} \},\\
\NC(n,\FF_q) &\to
  \{ \text{Irreducible supercharacters of $\UT_n(\FF_q)$} \}.
  \ea
  \ee 
Hence the number of supercharacters of $\UT_n(\FF_2)$ is the $n$th Bell number and the number of irreducible supercharacters is the $n$th Catalan number.  In addition, there is a useful closed form formula describing the values of $\chi_\Lambda$ in terms of the partition $\Lambda$; see \cite[Eq.\ 2.1]{Thiem}.

Seeking to emulate this combinatorial classification, in this work we study the Heisenberg characters of $\UT_n(\FF_q)$. Recall, for example from \cite[Appendix B]{BoyDrin}, that a character $\chi$ of a group $G$ is \emph{Heisenberg} if $\chi$ is irreducible and $\ker \chi \supset [G,[G,G]]$.  Of course, one naturally identifies the Heisenberg characters of $G$ with the irreducible characters of the quotient $G/[G,[G,G]]$.
The Heisenberg characters of $\UT_n(\FF_q)$ and related groups take the role usually played by linear characters in several inductive constructions introduced in \cite{BoyDrin}, which provides some motivation for examining such objects.
We also consider  characters of $\UT_n(\FF_q)$ which are invariant under multiplication by the  subgroup of linear characters
$C \omdef= \left\{ \vartheta \circ \hom : \vartheta \in \Hom\(\FF_q^+, \CC^\times\) \right\}$,
 where $\hom : \UT_n(\FF_q)\to \FF_q$ is the homomorphism which assigns to $g$   the sum of its entries on the first superdiagonal: \[\hom (g)\omdef = g_{1,2} + g_{2,3} + \dots+ g_{n-1,n},\qquad\text{for $g \in \UT_n(\FF_q)$}.\]
Our investigation of  $C$-invariant characters will allow us to count various families of characters of the interesting subgroup $\UT^\hom_n(\FF_q) \omdef =\{ g \in \UT_n(\FF_q) : \hom(g) = 0\}$, whose supercharacters are described in \cite[Example 5.1]{M_normal}.

Our main result is a classification of the Heisenberg characters of $\UT_n(\FF_q)$ in terms of certain labeled lattice paths.
Fix a subset $S \subset \Nat^2$, where $\Nat = \{0,1,2,\dots\}$ is the set of nonnegative integers.  
Formally, a \emph{lattice path} with steps in $S$ shall refer to a finite sequence $P = (s_1,s_2,\dots,s_k)$ of elements of $S$.  
The sequence $P$ represents the path which begins at $(0,0)$ and travels to $s_1+s_2+\dots + s_t \in \Nat^2$ at time $t \in [k]$, ending at the point $(a,b) = s_1+s_2+\dots +s_k$.
The most obvious way of labeling a lattice path would be to assign labels to each step, but instead we require the following slightly more subtle definition.

\begin{definition}
An \emph{$\FF_q$-labeled lattice path} is a lattice path with a rule assigning to each step $(i,j)$ in the path a sequence of labels $(x_1,\dots,x_j) \in (\FF_q^\times)^j$.
Thus, the number of nonzero labels assigned to a step is the step's height, and we consider steps of the form $(i,0)$ to be unlabeled.    
 \end{definition}

For each integer $n$, let $\Lattice(n,\FF_q,S)$ be  the set of $\FF_q$-labeled lattice paths with steps in $S\subset\Nat^2$ which end on the line $\{ (x,y) : x+y=n-1\}$.
 As special cases, we define
\begin{enumerate}
\item[] $\sK(n,\FF_q) \omdef=\Lattice\(n,\FF_q,S\)$ where $S = \{ (1,0),(1,1),(0,1)\}$,
\item[] $\sL(n,\FF_q) \omdef=\Lattice\(n,\FF_q,S'\)$ where $S'=\{ (1,0),(1,1),(0,1),(0,2)\}$,
\item[] $\sM(n,\FF_q) \omdef=\Lattice\(n,\FF_q,S''\)$ where $S''=\{ (2,1),(1,2),(0,1)\}.$
\end{enumerate}
We also define two modified families of labeled lattice paths 
\begin{enumerate}
\item[] $\sLL(n,\FF_q) \omdef= \Big\{\text{Paths in $\sL(n,\FF_q)$ which do not begin with the step $(0,2)$}\Big\}$,
\item[]
$\sMM(n,\FF_q) \omdef= \Big\{\text{Nonempty paths in $\sM(n,\FF_q)\cup\sM(n-1,\FF_q)$ beginning with $(0,1)$}\Big\}$.
\end{enumerate}
The relevance of these definitions derives from the following result, which we prove as Theorems \ref{heis-thm}, \ref{del-thm}, and \ref{c-heis-thm} below:

\begin{theorem}\label{thm1}
For each integer $n\geq 1$ there are bijections
\begin{enumerate}
\item[(1)] $\sK(n,\FF_q) \leftrightarrow \{\text{Heisenberg supercharacters of $\UT_n(\FF_q)$} \}$;
\item[(2)] $\sLL(n,\FF_q) \leftrightarrow \{\text{Heisenberg characters of $\UT_n(\FF_q)$} \}$;
\item[(3)] $\sMM(n,\FF_q) \leftrightarrow \{\text{$C$-invariant Heisenberg characters of $\UT_{n+1}(\FF_q)$} \}$.
\end{enumerate}
\end{theorem}

Our proof of this theorem derives essentially from the method of decomposing the supercharacters of an algebra group described in \cite{M}.
We compute from these correspondences explicit polynomial formulas for the numbers of Heisenberg supercharacters, Heisenberg characters, and $C$-invariant Heisenberg characters of the unitriangular group. 
Notably, the number of Heisenberg supercharacters of $\UT_n(\FF_q)$ is a polynomial in $q-1$ whose coefficients are  Delannoy numbers and whose value at $q=2$ is the $n$th Pell number.  Concerning the Heisenberg characters of $\UT_n(\FF_q)$, we have the following theorem:


\begin{theorem}\label{thm2}
For integers $n\geq 1$ and $e\geq 0$,
the number of Heisenberg characters of $\UT_{n+1}(\FF_q)$ 
with degree $q^e$ is 
\[  q^{n-e-1}\(\binom{n-e}{e}(q-1)^e + \binom{n-e-1}{e}(q-1)^{e+1}\).\]
Hence the number of all Heisenberg characters of $\UT_{n+1}(\FF_q)$ is a polynomial in $q-1$ with nonnegative integer coefficients and degree $n$, whose leading coefficient is the $n$th Fibonacci number.
\end{theorem}

Finally, using Clifford theory with these classifications, we are able to produce formulas proving  following result, given below as Theorem \ref{alt-thm}.

\begin{theorem}\label{thm3}
The numbers of supercharacters, irreducible supercharacters, Heisenberg supercharacters, and Heisenberg characters of the subgroup $\UT^\hom_n(\FF_q) \subset \UT_n(\FF_q)$ are polynomials in $q-1$ with nonnegative integer coefficients. 
\end{theorem}

We organize this article as follows.  In Section \ref{prelim-sect} we discuss in slightly greater depth the supercharacters of an algebra group, showing in particular how our brief definition above coincides with the standard notions developed in \cite{Andre,Andre1,DI,Yan}.  We then review the method of decomposing supercharacters into constituents described in \cite{M}, and prove in Section \ref{prep-sect} that such constituents exhaust all Heisenberg characters of $\UT_n(\FF_q)$.  In Sections \ref{heis-sect} and \ref{c-heis-sect} we classify the Heisenberg characters of $\UT_n(\FF_q)$, in particular proving Theorems \ref{thm1} and \ref{thm2}. Finally, in Section \ref{alt-sect}, we count the supercharacters, irreducible supercharacters, Heisenberg supercharacters, and Heisenberg characters of $\UT^\hom_n(\FF_q)$ in order to prove Theorem \ref{thm3}.

\section{Preliminaries}
\label{prelim-sect}


Here we review some of the properties of the supercharacters mentioned in the introduction.

\subsection{Supercharacters of algebra groups}\label{superchar-defs}

Our reference for the following statements is \cite{DI}.
As in the introduction, let $\fkn$ be a nilpotent $\FF_q$-algebra which is associative and finite-dimensional, and   write $G = 1+\fkn$ for the corresponding {algebra group} of formal sums $1+X $ with $X \in \fkn$.
Recall  that 
 $\theta : \FF_q^+ \to \CC^\times$ is a fixed nontrivial homomorphism.  
 
 The distinct irreducible characters of the additive group of $\fkn$ are the functions $\theta\circ \lambda$ with $\lambda$ ranging over all elements of the dual space $\fkn^*$.  Hence, if we define   $\theta_\lambda : G \to \CC^\times$ for $\lambda \in \fkn^*$ by
\[ \theta_\lambda(1+X) \omdef = \theta\circ\lambda(X)\text{ for }X \in \fkn\]
then the maps $\{ \theta_\lambda : \lambda \in \fkn^* \}$ form an orthonormal basis with respect to the usual $L^2$-inner product $\langle \cdot,\cdot\rangle$ of the vector space of functions $G \to \CC$.  
Observe that in the notation of the introduction we have $v_\lambda =\sum_{g \in G} \theta_{\lambda}(g)g \in \CC G$.
%

\begin{definition} \label{sch-def}
 Given $\lambda \in \fkn^*$, denote by $\chi_\lambda$ the \emph{complex conjugate} of the character of the left $G$-module $\CC G v_{\lambda} = \CC\spanning \{ g v_{\lambda} : g\in G\}$.
  \end{definition}
 
 \begin{remark}
Since $\chi_\lambda$ is the character of the module $\CC G v_{-\lambda}$, each $\chi_\lambda$ is a supercharacter according to the definition in the introduction and the set of supercharacters of $G$ is precisely $\{ \chi_\lambda : \lambda \in \fkn^*\}$. The complex conjugation in this definition is a notational convention meant to accommodate the formula (\ref{superchar-def}) below.
\end{remark}
 
 The group $G$ acts on the left and right on 
 $\fkn^*$ by $(g , \lambda) \mapsto g\lambda$ and $( \lambda,g) \mapsto \lambda g$ where we define
 \[ g\lambda(X) = \lambda(g^{-1}X)\qquad\text{and}\qquad \lambda g(X) = \lambda(Xg^{-1}),\qquad\text{for }\lambda \in \fkn^*,\ g \in G,\ X \in \fkn.\]
 These actions commute, in the sense that $(g\lambda) h = g(\lambda h)$ for $g,h \in G$, so there is no ambiguity in removing all parentheses and writing expressions like $g\lambda h$.   We denote the left, right, and two-sided orbits of $\lambda \in \fkn^*$ by  $G\lambda$, $\lambda G$, and $G\lambda G$; then $G\lambda$ and $\lambda G$ have the same cardinality and 
$ |G\lambda G| = \frac{|G\lambda||\lambda G|}{|G\lambda \cap \lambda G|}$ by \cite[Lemmas 3.1 and 4.2]{DI}.
A simple calculation shows
$gv_\lambda \in \CC\spanning\{ v_{g\lambda}\}$ and $v_\lambda g \in \CC\spanning\{ v_{\lambda g}\}$ for $g \in G$.
%
  It follows that $\CC G v_\lambda \cong \CC G v_\mu$ if $\mu \in G\lambda G$ and that $\CC G$ is the direct sum of the distinct two-sided ideals     $\CC G v_\lambda \CC G = \CC\spanning\{ v_\mu : \mu \in G\lambda G\}$.
  Consequently, $\chi_\lambda = \chi_\mu$ if $\mu \in G\lambda G$ and $\langle \chi_\lambda,\chi_\mu\rangle = 0$ if $\mu \notin G\lambda G$.  This means that the number of supercharacters of $G$ is the number of two-sided $G$-orbits in $\fkn^*$, and that each irreducible character of $G$ is a constituent of exactly one supercharacter.
 
  Since $\frac{1}{|G|} \sum_{\lambda \in \fkn^*} v_\lambda = 1$, the element $\frac{1}{|G|} \sum_{\mu \in G\lambda G} v_\mu$ is a central idempotent for $\CC G v_\lambda \CC G$.  By an elementary result in character theory (e.g., \cite[Proposition 14.10]{JL}), the coefficient of $g \in G$ in this idempotent is the complex conjugate of the value of the character of $\CC G v_\lambda \CC G$ at $g$. This character is just $\overline{\chi_\lambda}$ times $|G\lambda G|/|G\lambda|$, the number of left $G$-orbits in $G \lambda G$, and so
\be\label{superchar-def} 
 \chi_\lambda = \frac{|G\lambda|}{|G\lambda G|} \sum_{\mu \in G \lambda G} \theta_\mu.
\ee 
The orthogonality of the functions $\{ \theta_\mu\}$ implies that $\langle \chi_\lambda,\chi_\lambda \rangle = \frac{|G\lambda|^2}{|G\lambda G|} = |G\lambda \cap \lambda G|$.  Thus, $\chi_\lambda$ is irreducible if and only if $G\lambda \cap \lambda G = \{\lambda\}$.  As $\frac{|G\lambda G|}{|G\lambda|} \chi_\lambda$ is the character of the two-sided ideal $\CC G v_\lambda \CC G$, an irreducible constituent $\psi$ of $\chi_\lambda$ appears with multiplicity $\frac{|G\lambda|}{|G\lambda G|} \psi(1)$.  

\subsection{Decomposing supercharacters}
\label{xi}

\def\rad{\mathrm{rad}}
\def\radl{\ker_{\mathrm{L}}}
\def\radr{\ker_{\mathrm{R}}}

In practice, one can often decompose the set of supercharacters of an algebra group into a larger family of orthogonal characters by adopting a slightly more intricate construction.  This ability will prove useful in classifying the Heisenberg characters of $\UT_n(\FF_q)$,  so we briefly describe it here.

Proofs of the following statements appear in \cite{M}.  For each $\lambda \in \fkn^*$, define two sequences of subspaces $\fk l_\lambda^i, \fk s_\lambda^i \subset \fkn$ for $i\geq 0$ by the inductive formulas 
\[ \ba \fk l_\lambda ^0 &= \{0\}, \\ 
\fk s_\lambda^0 &= \fkn, \ea \qquad \text{and}\qquad 
\ba \fk l_\lambda^{i+1} & = \left\{ X \in \fk s_\lambda^i : \lambda(XY) = 0 \text{ for all }Y \in \fk s_\lambda^i \right\}, \\ 
\fk s_\lambda^{i+1} & =\left \{ X \in \fk s_\lambda^i : \lambda(XY) = 0 \text{ for all }Y \in \fk l_\lambda^{i+1}\right \}. \ea\]
One can show that 
$ \chi_\lambda(1)= |G\lambda|= |\fkn| / |\fk l _\lambda^1|$ and $\langle \chi_\lambda,\chi_\lambda\rangle =|G\lambda\cap\lambda G|=  |\fk s_\lambda^1| / |\fk l_\lambda^1|.$
In addition, the subspaces $\fk l_\lambda^i$, $\fk s_\lambda^i$ are all subalgebras of $\fkn$ satisfying the following chain of inclusions:
\[\label{chain} 0  =\fk l_\lambda^0 \subset  \fk l_\lambda^1 \subset \fk l_\lambda^2 \subset \cdots\subset \fk s_\lambda^2 \subset \fk s_\lambda^1  \subset \fk s_\lambda^0 =  \fkn.\]
These chains eventually stabilize since our algebras are finite-dimensional; we therefore let $\olfkl_\lambda = \bigcup_i \fk l_\lambda^i$ and $\olfks_\lambda = \bigcap_i \fk s_\lambda^i$ and define $\oll_\lambda = 1+\olfkl_\lambda$ and $\ols_\lambda = 1+ \olfks_\lambda$ as the corresponding algebra subgroups of $G$.
The function $\theta_\lambda: g \mapsto \theta\circ\lambda(g-1)$ restricts to a linear character of $\oll_\lambda$ and we define $\xi_\lambda$ as the corresponding induced character of $G$:
\[\xi_\lambda = \Ind_{\oll_\lambda}^G(\theta_\lambda).\]
This character is a possibly reducible constituent  of the supercharacter $\chi_\lambda$; these are its key properties:
\begin{enumerate}
\item[(i)] Each supercharacter $\chi_\lambda$ for $\lambda \in \fkn^*$ is a linear combination of the characters $\{\xi_\mu : \mu \in G\lambda G\}$ with positive integer coefficients.
 
\item[(ii)] Either $\xi_\lambda = \xi_\mu$ or $\langle \xi_\lambda,\xi_\mu\rangle = 0$ for each $\lambda,\mu \in \fkn^*$.

\item[(iii)] We have $\langle\xi_\lambda,\xi_\lambda\rangle = |\ols_\lambda|/|\oll_\lambda|$, and so $\xi_\lambda$ is irreducible if and only if $\oll_\lambda = \ols_\lambda$.  When $\xi_\lambda$ is irreducible, it has the formula 
\be\label{xi-formula}\xi_\lambda = \frac{1}{\sqrt{|\lambda^G|}} \sum_{\mu \in \lambda^G} \theta_\mu,\qquad\text{where }\lambda^G \omdef=\left \{ g\lambda g^{-1} : g\in G\right\}.\ee  The set $\lambda^G$ is the \emph{coadjoint orbit} of $\lambda$.  The right hand side of this formula is a well-defined function $G \to \CC$ (though not necessarily a character)  for all $\lambda \in \fkn^*$; functions on algebra groups with this form are typically called  \emph{Kirillov functions}.

\end{enumerate}

\section{Constructions and lemmas}
\label{prep-sect}

In this section we prove that the characters $\{\xi_\lambda\}$ introduced in Section \ref{xi} include all Heisenberg characters of $\UT_n(\FF_q)$.  Our results  will put us in good standing to  enumerate the Heisenberg characters of $\UT_n(\FF_q)$ in Section \ref{heis-sect}.

Our first lemma applies to all algebra groups.  

\begin{lemma}\label{heis-lem1} Fix a nilpotent $\FF_q$-algebra $\fkn$ which is finite-dimensional and associative, let $G = 1+\fkn$ be the corresponding algebra group, and choose $\lambda \in \fkn^*$ such that $\ker \lambda \supset \fkn^3$.  If $\ols_\lambda = \oll_\lambda$, then the supercharacter $\chi_\lambda$ is equal to $|\fk s_\lambda^1|/|\olfks_\lambda|$ times a sum of $\frac{|\olfks_\lambda|^2}{|\fk l_\lambda^1| |\fk s_\lambda^1|} $ 
 distinct irreducible characters of degree $|\fkn |/|\olfks_\lambda|$.  The set of these irreducible constituents is given by $\{ \xi_\mu : \mu \in G\lambda G\}$. 
 \end{lemma}

\begin{proof}
Let $\mu \in G\lambda G$.  If $X \in \fkn^2$ then $gXh-X \in \fkn^3$ for all $g,h \in G$.  Since $\ker \lambda \supset\fkn^3$, it follows that $\lambda(XY) = \mu(XY)$ for all $X,Y \in \fkn$, and from this it is clear that $\ols_\lambda = \ols_\mu$ and $\oll_\lambda = \oll_\mu$.   Assuming $\ols_\lambda = \oll_\lambda$, it follows by property (i) in Section \ref{xi}  that every irreducible constituent of $\chi_\lambda$ is a character $\xi_\mu$ for some $\mu \in G\lambda G$, and in particular, these constituents all have degree $|G|/|\oll_\lambda|$.  From the discussion at the end of Section \ref{superchar-defs}, it follows that these irreducible constituents each appear in $\chi_\lambda$ with multiplicity $\frac{|G\lambda|}{|G\lambda G|} \frac{|G|}{|\oll_\lambda|} = \frac{|\fk s_\lambda^1|}{|G|} \frac{|G|}{|\ols_\lambda|} = |\fk s_\lambda^1|/|\olfks_\lambda|$.  Dividing $\chi_\lambda(1) = |G|/|\fk l_\lambda^1|$ by the product of this multiplicity with $\xi_\lambda(1) = |G|/|\oll_\lambda| = |G| / |\olfks_\lambda|$, we find the $\chi_\lambda$ must have $\frac{|\olfks_\lambda|^2}{|\fk l_\lambda^1| |\fk s_\lambda^1|} $ distinct irreducible constituents.
\end{proof}

In the following lemma we use the term  \emph{pattern algebra} to mean a subalgebra of  $\fkt_n(\FF_q)$ for some $n$ and $q$ spanned by a subset of the elementary matrices $\{ e_{ij} : 1\leq i<j \leq n\}$. Likewise, a  \emph{pattern group} shall mean an algebra group corresponding to a pattern algebra.  Equivalently, a pattern group is the group of unipotent elements of the incidence algebra of a finite poset over a finite field.
Pattern groups form a highly accessible family of algebra groups and serve as fundamental examples in supercharacter theory; for background and discussion, see \cite{DT}.  Our second lemma is a straightforward observation concerning the lower central series and Heisenberg characters of pattern groups.  Here we employ the notation $[g,h] = ghg^{-1}h^{-1}$ for $g,h \in G$, and for subsets $S,T\subset G$ write $[S,T]$ to denote the group generated by the commutators $[g,h]$ for $g \in S$ and $h \in T$.
Also, $\Mat_n(\FF)$ denotes the set of $n\times n$ matrices over a field $\FF$.

\begin{lemma}\label{heis-obs} Let $\fkn$ be a pattern algebra with $G = 1+\fkn$ the corresponding pattern group.
\begin{enumerate}
\item[(i)] Let $G_1=G$ for $i\geq 1$ and set $G_{i+1} =  [G_i,G]$.  Then $G_i = 1 + \fkn^i$.
%

\item[(ii)] If $\fkn = \fkt_n(\FF_q)$ and $k$ is a positive integer, then $\fkn^k = \left\{ X \in \Mat_n(\FF_q) : X_{ij} = 0\text{ if } j < i+k\right\}$.

\item[(iii)] A character $\psi \in \Irr(G)$ is Heisenberg if and only if $\ker \psi \supset 1+\fkn^3$, which occurs if and only if $\psi$ is a constituent of a supercharacter $\chi_\lambda$ indexed by a functional $\lambda \in \fkn^*$ with $\ker \lambda \supset \fkn^3$.

\end{enumerate}
\end{lemma}

Isaacs proves the special case $[G,G] = 1+\fkn^2$ of the first part as \cite[Corollary 2.2]{I}.

\begin{proof}
 The inclusion $[1+\fkn^i,G] \subset 1+\fkn^{i+1}$ holds for all algebra groups $G$ and is straightforward to check.  For the reverse direction, embed $\fkn \subset \fkt_n(\FF_q)$ and suppose the elementary matrix $e_{j\ell} \in \fkn^{i+1}$.  Because pattern algebras are spanned by elementary matrices, there must exist $k$ with $e_{jk} \in \fkn$ and $e_{k\ell} \in \fkn^i$.  Necessarily $j<k<\ell$ and we have $1+te_{j\ell} = [1+e_{jk}, 1+te_{k\ell} ] \in [G,G_i]$ for all $t \in \FF_q$.  Elements of the form $1+t e_{j\ell}$ generate the pattern group $1+\fkn^{i+1}$ (see the discussion in Section 3.2 of \cite{DT}), so our inclusion in an equality.

Part (ii) is a simple calculation, and the first statement in (iii) is immediate from (i).  The claim that $\psi \in \Irr(G)$ has $\ker \psi \supset 1+\fkn^3$ if and only if $\psi$ is a constituent of a supercharacter $\chi_\lambda$ with $\ker \lambda \supset \fkn^3$ follows from \cite[Proposition 3.3]{M_normal} and standard results in character theory concerning inflation.
 \end{proof}

We now have our promised result.

\begin{proposition}\label{heis-cor}
Each Heisenberg character of $\UT_n(\FF_q)$ is equal to $\xi_\lambda$ for some $\lambda \in \fkt_n(\FF_q)^*$.
\end{proposition}

\begin{proof}
Each Heisenberg character is a constituent of a supercharacter $\chi_\lambda$ for some $\lambda \in \fkt_n(\FF_q)^*$ with $\ker \lambda \supset (\fkt_n(\FF_q))^3$.  By  (\ref{bell-thm}), there exists $\mu$ in the two-sided $\UT_n(\FF_q)$-orbit of $\lambda$ which is ``quasi-monomial'' in the sense \cite[Section 3.1]{M_X} and which consequently has $\oll_\mu = \ols_\mu$ by \cite[Theorem 3.1]{M_X}.  By Lemma \ref{heis-lem1},  the irreducible constituents of $\chi_\lambda=\chi_\mu$ are therefore the characters $\xi_\nu$ for the maps $\nu$ in the two-sided $\UT_n(\FF_q)$-orbit of $\lambda$. 
\end{proof}

In preparation for Sections \ref{heis-sect} and \ref{c-heis-sect}, we now prove two somewhat more technical results.
Recall that if $\fkn$ is a nilpotent $\FF_q$-algebra and $G=1+\fkn$, then 
the coadjoint orbit of $\lambda \in \fkn^*$ is the set of functionals $\{ g\lambda g^{-1}  : g \in G\}$.
We write $e_{ij}^* \in \fkt_n(\FF_q)^*$ for the linear functional with $e_{ij}^*(X) = X_{ij}$ for $1\leq i<j\leq n$.  

\begin{lemma}\label{tech-lem1} 
Choose integers $j,k$ with $1 \leq j \leq k \leq n-2$, and let $
\lambda_t = \sum_{ i=j}^k t_i e_{i,i+2}^*
\in \fkt_n(\FF_q)^*$ for some 
$t=(t_j,\dots,t_k) \in (\FF_q^\times)^{k-j+1}$.
Assume $k-j$ is odd and let $d=(k-j+1)/2$.  The following statements then hold:
\begin{enumerate}

\item[(1)] The coadjoint orbit of $\lambda_t$ is the set 
\[ \left\{ \lambda_t + \sum_{i=0}^{2d} u_i e_{i+j,i+j+1}^* : 
u_1,\dots,u_{2d}\in \FF_q\text{ and } u_0 = \sum_{i=1}^d a_iu_{2i} \right\}\] where  $a_1,a_2,\dots,a_d \in \FF_q^\times$ are elements determined by $t$.

\item[(2)] The coadjoint orbits of $\lambda_t + u e_{j,j+1}^*$ for $u \in \FF_q$ are pairwise disjoint.

\item[(3)] Let $\gamma = e_{j,j+1}^*+e_{j+1,j+2}^* +\dots + e_{k,k+1}^* + e_{k+1,k+2}^* \in \fkt_n(\FF_q)^*$.  Then 
%
the cardinality of the set 
\[ \left\{ t \in (\FF_q^\times)^{2d} :  \text{$\lambda_t + \gamma$ belongs to the coadjoint orbit of $\lambda_t$
} 
\right\}
\]
is $\frac{1}{q}\((q-1)^{2d}- (-1)^d(q-1)^d\)$.

\end{enumerate}
\end{lemma}


\begin{proof}
Direct calculation shows that every element of the coadjoint orbit of $\lambda_t$ has the form $\lambda_t + \kappa$ for some $\kappa \in \FF_q\spanning\{ e_{i,i+1}^* :  j\leq i \leq k+1 \}$, and that, since every such $\kappa$ is invariant under the two-sided action of $\UT_n(\FF_q)$, one has $g^{-1} \lambda_t g = \lambda_t+\kappa$ for $g \in \UT_n(\FF_q)$ if and only if $g^{-1}\lambda_t - \lambda_t g^{-1} = \kappa$.
As an abbreviation, define $g_i = g_{i,i+1}$.  Then the map $g\mapsto g^{-1}\lambda_t -\lambda_t g^{-1}$ is a linear function of the entries $g_{j}$, $g_{j+1}$, \dots, $g_{k}$, $g_{k+1}$;
in particular, 
we compute $g^{-1}\lambda_t -\lambda_t g^{-1} = \sum_{i=0}^{2d} u_i e_{i+j,i+j+1}^*$ where 
\be\label{linsys}
\(\barr{lllll}
 0 & -t_{j} \\
   t_{j} & 0 & \ddots \\
    & t_{j+1}& \ddots & -t_{k-1}\\
    && \ddots & 0 & -t_{k} \\
&&&  t_{k} & 0 
  \earr\)\(\barr{l} g_{j} \\ g_{j+1} \\ \vdots  \\ g_{k} \\ g_{k+1} \earr\)
  = \(\barr{l} u_0 \\ u_{1} \\  \vdots  \\ u_{2d-1} \\ u_{2d} \earr\).
  \ee 
  Define  linear polynomials  $P_{t,i}(x) = P_{t,i}(x_1,\dots,x_d)$ for $1\leq i\leq d$ by
  the recurrence \[ P_{t,1}(x) = x_1\qquad\text{and}\qquad P_{t,i+1}(x) =x_{i+1} + (t_{k-2i+1}/t_{k-2i+2})P_{i}(x)\text{ for }1\leq i< d,\] and set $P_t(x) = -(t_j/t_{j+1}) P_{t,d}(x)$.  
It is  a straightforward exercise  to show by induction that the linear system (\ref{linsys}) has a solution $g \in \UT_n(\FF_q)$ if and only if $u_0 =  P_{t}(u_{2d},u_{2d-2},\dots,u_2)$, and this suffices to confirm our description of the coadjoint orbit of $\lambda_t$.

%
  
If $\kappa,\kappa' \in \FF_q\spanning\{e_{j,j+1}^*\}$, then $\lambda_t+\kappa$ and $\lambda_t+\kappa'$ belong to the same coadjoint orbit if and only if $\lambda_t$ and $\lambda_t+(\kappa-\kappa')$ belong to the same coadjoint orbit.  Our claim in (2) follows immediately.  For the second part, let $f_d(q)$ denote the cardinality of the set in (3) and note that $\lambda_t +\gamma$ belongs to the coadjoint orbit of $\lambda_t$ if and only if $ P_{t}(1,\dots,1) = 1$.  
It follows immediately that $f_1(q)=q-1$.  Observe that if $t' = (t_{j+2},t_{j+3},\dots,t_k) \in (\FF_q^\times)^{2d-2}$ then $P_{t}(1,\dots,1) = (t_j/t_{j+1}) (P_{t'}(1,\dots,1) - 1)$.
Consequently  if $t'$ is fixed, then the number of choices for the pair $(t_{j}, t_{j+1}) \in (\FF_q^\times)^2$ such that $\lambda_t +\gamma$ belongs to the coadjoint orbit of $\lambda_t$ is 
$q-1$ if $P_{t'}(1,\dots,1) \neq 1$ and 0 otherwise.
Hence $f_d(q) = (q-1)\( (q-1)^{2d-2}-f_{d-1}(q)\)$ for $d> 1$ and $f_1(q)=q-1$,   from which we deduce that $f_d(q)$ is equal to the given expression.
\end{proof}

Almost as a corollary, we have this second lemma.

\begin{lemma}\label{tech-lem}
Choose integers $j,k$ with $1 \leq j \leq k \leq n-2$, and let $
\lambda = \sum_{ i=j}^k t_i e_{i,i+2}^*
\in \fkt_n(\FF_q)^*$ for some 
$t_j,\dots,t_k \in \FF_q^\times$.  The following properties then hold:
\begin{enumerate}
\item[(1)] If $k-j$ is even then
the coadjoint orbit of $\lambda$ is $\lambda +\FF_q\spanning\left\{ e_{i+j,i+j+1}^* :0 \leq i \leq k-j+1\right\}$ and
$\xi_\lambda$ is the unique irreducible constituent of $\chi_\lambda$.

\item[(2)] If $k-j$ is odd then $\chi_\lambda$ has exactly $q$ irreducible constituents, given by the characters $\xi_{\lambda+\kappa}$ for $\kappa \in \FF_q\spanning\{ e_{j,j+1}^* \}$.
\end{enumerate}
\end{lemma}

\begin{proof}
It is a straightforward if tedious exercise to show that 
\[ \ba \fk l_\lambda^1 &= \left\{ X \in \fkt_n(\FF_q) : X_{j,j+1} = X_{j+1,j+2} = \dots = X_{k,k+1} = 0\right\}, \\
\fk s_\lambda^1 & = \left\{ X \in \fkt_n(\FF_q) : X_{k,k+1} = 0\right\},
\ea
\] and 
\[ \olfkl_\lambda = \olfks_\lambda =
 \begin{cases}
\left\{ X \in \fkt_n(\FF_q) : X_{j,j+1} = X_{j+2,j+3} = \dots = X_{k,k+1} = 0\right\}, &\text{if $k-j$ is even}; 
\\
\left\{ X \in \fkt_n(\FF_q) : X_{j+1,j+2} = X_{j+3,j+4} = \dots = X_{k,k+1} = 0\right\}, &\text{if $k-j$ is odd}.
\end{cases}
\] 
From this, it follows by Lemma \ref{heis-lem1} that $\chi_\lambda$ has a unique irreducible constituent (necessarily given by $\xi_\lambda$) if $k-j$ is even and exactly $q$ irreducible constituents if $k-j$ is odd.  In the even case, the irreducible characters $\xi_\mu$ for $\mu$ in the two-sided orbit of $\lambda$ are all equal, so from the formula (\ref{xi-formula}) we deduce that the coadjoint orbit and  two-sided orbit of $\lambda$  coincide.  Checking that the two-sided orbit is $\lambda +\FF_q\spanning\left\{ e_{i+j,i+j+1}^* :0 \leq i \leq k-j+1\right\}$ is an exercise in linear algebra, similar to but much more straightforward than the proof of our previous lemma.

Assume $k-j$ is odd.  Then the $q$ characters $\xi_{\lambda+\kappa}$ for $\kappa \in \FF_q\spanning\{ e_{j,j+1}^* \}$ are all irreducible constituents of $\chi_\lambda$ since $\lambda + \kappa$ lies in the right $\UT_n(\FF_q)$-orbit of $\lambda$, so it suffices to show that these characters are distinct.   
Noting the formula (\ref{xi-formula}) above and the fact that the functions $\{\theta_\mu\}$ are linearly independent, we deduce that the irreducible characters $\xi_{\lambda+\kappa}$ are distinct if and only if
the coadjoint orbits of $\lambda+\kappa$ are pairwise disjoint, which was shown in the previous lemma.
\end{proof}

%
%
%

\section{Heisenberg characters and lattice paths}
\label{heis-sect}

In this section we use the preceding results to derive a bijection between the set of Heisenberg characters of $\UT_n(\FF_q)$ and the set of $\FF_q$-labeled lattice paths $\sLL(n,\FF_q)$. Leading to this, we first establish a few properties of the polynomials counting the five families of $\FF_q$-labeled lattice paths 
defined in the introduction.

Let 
  $\D(a,b)$, $\H(a,b)$, and $\I(a,b)$ denote the respective numbers of (unlabeled) lattice paths ending at $(a,b) \in \ZZ^2$ with steps in the sets 
  \[S = \{ (1,0),(1,1),(0,1)\},\quad S'=\{(1,0),(1,1),(0,1),(0,2)\},\quad\text{and}\quad S''=\{(2,1),(1,2),(0,1)\}.\] The integers $\D(a,b)$ are  the well-known {Delannoy numbers} by definition.  By construction, the polynomials 
\begin{enumerate}
\item[] $\ds\Del{n+1}{x} \omdef = \sum_{k=0}^n \D(n-k,k) x^{k}$,

\item[] $\ds\preHe{n+1}{x} \omdef = \sum_{k=0}^n \H(n-k,k) x^{k}$,

\item[] $\ds\preIn{n+1}{x}\omdef =  \sum_{k=0}^n \I(n-k,k) x^{k}$,

\end{enumerate}
 give the respective cardinalities of $\sK(n+1,\FF_q)$, $\sL(n+1,\FF_q)$, and $\sM(n+1,\FF_q)$, when $n\geq 0$ and $x=q-1$.  We set  $\Del{-n}{x} = \preHe{-n}{x} =\preIn{-n}{x} = 0$ for $n\geq 0$ so that this relationship holds for all integers $n$, and note 
the ordinary generating functions
\begin{enumerate}
\item[] $\ds\sum_{n\geq 0} \Del{n}{x} z^n= \frac{z}{1 - (x+1)z - xz^2},$

\item[] $\ds\sum_{n\geq 0} \preHe{n}{x} z^n = \frac{z}{1 - (x+1)z - x(x+1) z^2}$,

\item[] $ \ds\sum_{n\geq 0} \cI_n(x) z^n = \frac{z}{1-xz -x(x+1)z^3}.$
\end{enumerate}
The first of these suffices to show that $\{\Del{n}{1}\}_{n=0}^\infty=(0,1,2,5,12,29,70,\dots)$ is the sequence of \emph{Pell numbers}  \cite[A000129]{OEIS}, which explains our notation.

The following proposition gives more explicit formulas for these polynomials.

\begin{proposition}\label{preHe-formulas} For   integers $a,b,n\geq 0$, the following identities hold:
\begin{enumerate}
\item[(1)] $\ds
\Del{n+1}{x}  =  \sum_{k=0}^{\lfloor n/2\rfloor}
 \binom{n-k}{k} x^k (1+x)^{n-2k}$ and $\ds\D(a,b) = \sum_{k\geq 0} \binom{a+b-k}{k} \binom{a+b-2k}{b-k}$.

\item[(2)] $\ds
\preHe{n+1}{x} = \sum_{k=0}^{\lfloor n/2 \rfloor}
 \binom{n-k}{k} x^k (1+x)^{n-k}$ and $\ds\H(a,b) = \sum_{k\geq 0} \binom{k}{a+b-k} \binom{k}{a}$.

\item[(3)] $\ds
\preIn{n+1}{x} = \sum_{k=0}^{\lfloor n/3 \rfloor}
 \binom{n-2k}{k} x^{n-2k} (1+x)^k$ and $\ds\I(a,b) = \sum_{k\geq 0} \binom{a+b-2k}{k} \binom{k}{a-k}$.

\end{enumerate}
\end{proposition}

\begin{proof}
Supposing $x=q-1$, it is a routine exercise to show that the $k$th term in each of the three sums gives the number of $\FF_q$-labeled lattice paths ending on the line $x+y=n$ with $k$ steps in $R$ and the rest of its steps in $T$, where 
\[ R = \begin{cases} \{(1,1)\}, &\text{in part (1)}, \\
 \{(1,1), (0,2)\}, &\text{in part (2)}, \\
\{  (2,1), (1,2)\}, &\text{in part (3)},
\end{cases}
\qquad\text{and}\qquad 
T = \begin{cases} \{(1,0),(0,1)\}, &\text{in part (1)}, \\
 \{(1,0), (0,1)\}, &\text{in part (2)}, \\
\{  (0,1)\}, &\text{in part (3)}.
\end{cases}
\]
This observation is enough to establish our formulas for $\Del{n+1}{x}$, $\preHe{n+1}{x}$, and $\preIn{n+1}{x}$. Extracting the coefficients of $x^k$ in these sums gives our formulas  for  $\D(a,b)$, $\H(a,b)$, and $\I(a,b)$. \end{proof}

For each integer $n$, we additionally define \[ 
  \He{n}{x} \omdef = \preHe{n}{x} - x^2\preHe{n-2}{x}
 \qquad\text{and}\qquad
  \In{n}{x} \omdef = x\preIn{n-1}{x} + x\preIn{n-2}{x},
  \]
 so that 
 $|\sLL(n,\FF_q)| = \He{n}{q-1}$ and $|\sMM(n,\FF_q)| = \In{n}{q-1}$. Then $\He{1}{x} = 1$ and $\In{1}{x}  = 0$, and  the preceding proposition with the standard recurrence for the binomial coefficients implies the following formulas for $n\geq 2$.  
 
\begin{corollary}\label{He-formulas} For each  integer $n\geq 1$, the following identities hold:
\begin{enumerate}
\item[(1)] $\ds\He{n+1}{x} = \sum_{k=0}^{\lfloor n/2 \rfloor} \(\binom{n-k}{k} + \binom{n-k-1}{k}x \) x^k (1+x)^{n-k-1}$.
\item[(2)]
$\ds\In{n+1}{x} = \sum_{k=0}^{\lfloor (n-1)/3 \rfloor} \(\binom{n-2k-2}{k} + \binom{n-2k-1}{k}x\) x^{n-2k-1} (1+x)^k.$
\end{enumerate}
\end{corollary}

Thus  $\He{n}{x}$ and $\In{n}{x}$ both have nonnegative integer coefficients.
The preceding results show that   the leading coefficient of $\preHe{n+1}{x}$ and $\He{n+2}{x}$ is $\sum_k\binom{n-k}{k}$, a well-known formula for the $(n+1)$th Fibonacci number.  We note this briefly as a  second corollary.

\begin{corollary}\label{fib}  Let $\{f_n\}_{n=0}^\infty = (0,1,1,2,3,5,8,\dots) $ denote the Fibonacci numbers.  
Then $f_n$ is
the leading coefficient of  $\preHe{n}{x}$ and $\He{n+1}{x}$  for $n\geq 1$. 
\end{corollary}

\begin{remark}
The second-to-leading coefficients of these polynomials, counting lattice paths ending on the line $x=1$, are more obscure but have also appeared in other places. 
For example, the coefficients 
 $\{ \H(1,n)\}_{n=0}^\infty = (1,3,7,15,30,58,109,\dots)$ of $x^n$ in $\preHe{n+2}{x}$ form the sequence \cite[A023610]{OEIS} given by convolving the Fibonacci numbers with a shifted version of themselves.
On the other hand, the coefficients  $\{ \H(1,n)-\H(1,n-2)\}_{n=0}^\infty 
 = (1,3,6,12,23,43,79,\dots)$ of $x^n$ in $\He{n+2}{x}$
provide a sequence  \cite[A055244]{OEIS} counting certain stackings of $n$ squares on a double staircase.  Turban calls such configurations ``lattice animals'' and studies their enumeration in \cite{Turban}.  It is a reasonably straightforward exercise to show using \cite[Eq.\ (3.9)]{Turban} that the numbers which Turban labels $G_k$ have $G_{k+1} = \H(1,k) - \H(1,k-2)$ for $k\geq 0$. 
\end{remark}

To establish the promised correspondence between elements of $\sLL(n,\FF_q)$ and Heisenberg characters of $\UT_n(\FF_q)$, we use the symbols $\r$, $\d$, $\u$, $\uu$ to denote
 the steps $(1,0)$, $(1,1)$, $(0,1)$, $(0,2)$
and we indicate labels with subscripts.  Thus, we may list the paths in $\sL(4,\FF_q)$ 
ending at $(0,3)$ as  $(\uu_{x,y}, \u_z)$, $(\u_x, \uu_{y,z})$,  $(\u_x, \u_y, \u_z)$
where $x,y,z \in \FF_q^\times$ are arbitrary; note that the path  $(\uu_{x,y}, \u_z)$ is excluded from $\sLL(4,\FF_q)$.

Define a  map $\sLL(n,\FF_q) \to \fkt_n(\FF_q)^*$ in the following way.
Fix a labeled path
\[ P = (s_1,s_2,\dots,s_\ell) \in \sLL(n,\FF_q).\]  
Let $d_k(P)$ be the sum of the coordinates of  the starting point of the $k$th step of $P$, so that this starting point lies on the line $x+y = d_k(P)$.  Now, 
define $\lambda_{P,k} \in \fkt_n(\FF_q)^*$ for $k=1,\dots,\ell$ by
\[\lambda_{P,k}  = 
\begin{cases} 
0, &\text{if $s_k =\ \r$};\\
t e_{i,i+1}^*, &\text{if $s_k =\ \u_t$ for some $t \in \FF_q^\times$, where $i=d_k(P)+1$};\\
t e_{i,i+2}^*, &\text{if $s_k =\ \d_t$ for some $t \in \FF_q^\times$, where $i=d_k(P)+1$};\\
t e_{i-1,i+1}^* + u e_{i,i+2}^*, &\text{if $s_k =\ \uu_{t,u}$ for some $t,u \in \FF_q^\times$, where $i=d_k(P)+1$}.
\end{cases}
\]
Finally let $\lambda_P \in \fkt_n(\FF_q)^*$ be the sum 
$ \lambda_P = \lambda_{P,1} + \lambda_{P,2} + \dots +\lambda_{P,\ell}$. 
As an example, if 
\[ P = (\r,\d_a,\u_b,\uu_{c,d})\qquad\text{then}\qquad \lambda_P = a e_{2,4}^* + b e_{4,5}^* + ce_{4,6}^* + d e_{5,7}^*.\] 
Recall the definition of the characters $\xi_\lambda$ in Section \ref{xi}, and write $\xi_P$ for the character indexed by $\lambda_P \in \fkt_n(\FF_q)^*$ for $P \in \sLL(n,\FF_q)$.  
We now assert the following.

\begin{theorem}\label{heis-thm} 
The map  
\[\barr{ccc} \sLL(n,\FF_q) &\to
&
  \bigl\{ \text{Heisenberg characters of $\UT_n(\FF_q)$}\bigr \}  \\
P & \mapsto &  \xi_{P}
\earr\]
is a bijection, and the number of Heisenberg characters of $\UT_n(\FF_q)$
 is $\He{n}{q-1}$.
\end{theorem}

\begin{remark} The sequence $\left\{ \He{n}{1} \right\}_{n=1}^\infty=(1,2,5,14,38,104, 284,\dots)$ counting the Heisenberg characters of $\UT_n(\FF_2)$ also counts the number of compositions of $n-1$ whose odd parts are labeled in one of  two different ways \cite[A052945]{OEIS}.
\end{remark} 

 
 Before proceeding, we must introduce a few technical definitions.  First, define the \emph{upper form} of a matrix $X \in \fkt_n(\FF_q)$ to be the $(n-1)$-by-$(n-1)$ matrix given by deleting the first column and last row of $X$.  The upper form uniquely determines $X \in \fkt_n(\FF_q)$ because the row and column deleted contain all zeros.   Next, define the \emph{matrix} of a functional $\lambda \in \fkt_n(\FF_q)^*$ to be the  matrix $X = \sum_{1\leq i<j\leq n} \lambda(e_{ij}) e_{ij}$; this is the unique matrix $X \in \fkt_n(\FF_q)$ for which $\lambda(Y) = \tr(X^T Y)$.
 
 We now define the \emph{block decomposition} of $\lambda \in \fkt_n(\FF_q)^*$ as the maximal sequence of square matrices $(B_1,\dots,B_\ell)$ such that the upper form of the matrix of $\lambda$ is $\diag(B_1,\dots,B_\ell)$.
 As an example, the block decomposition of $\lambda=0$ is $(B_1,\dots,B_{n-1})$ where each $B_i$ is the 1-by-1 zero matrix, and 
if
 $r,s,t,u,v \in \FF_q^\times$ then the block decomposition of 
\be\label{lambda-ex}\lambda = r e_{1,3}^* + s e_{4,5}^* + t e_{4,6}^* + u e_{5,7}^* + v e_{7,8}^*  = \(\barr{cccccccc}
0& 0& r \\
&0 &0 \\
& & 0 & 0\\
& & & 0 & s & t & 0 \\
& & & & 0 &0 & u \\
& & & & & 0 & 0 \\
& & & & & & 0 & v \\
& & & & & & & 0 
\earr\) \in \fkt_8(\FF_q)^*\ee is $(B_1,B_2,B_3,B_4)$ where 
\[ B_1 = \(\barr{cc} 0 & r\\ 0 & 0 \earr\),\qquad B_2 = \(\barr{c} 0 \earr\),\qquad B_3 = \(\barr{ccc} s & t & 0 \\ 0 & 0 & u \\ 0 & 0 & 0 \earr\),\qquad B_4 = \( \barr{c} v \earr\).\]

\begin{proof}[Proof of Theorem \ref{heis-thm}]
Let $\cX$  be the set of functionals $\lambda \in \fkt_n(\FF_q)^*$ such that each matrix in the block decomposition of $\lambda$ is of one of the following types:
\begin{enumerate}
\item[(a)] $X$ is 1-by-1.

\item[(b)] $X$ is $m$-by-$m$ for some $m>1$ and $X_{ij} \neq 0$ if and only if $j=i+1$.

\item[(c)] $X$ is $m$-by-$m$ for some  $m>1$ with $n$ odd and $X_{ij} \neq 0$ if and only if $j=i+1$ or $i=j=1$.  
\end{enumerate}
As the first step of our proof, we show that $P \mapsto \lambda_P$ defines a bijection $\sLL(n,\FF_q) \to \cX$.  It is tedious but not difficult to check that $\lambda_P \in \cX$ for all $P \in \sLL(n,\FF_q)$.  We proceed by constructing an inverse map $\cX \to \sLL(n,\FF_q)$.  To this end, associate to each $\lambda \in \cX$ the $\FF_q$-labeled lattice path $P_\lambda$ defined by the following algorithm:

\begin{enumerate}
\item[(1)] Suppose $\lambda \in \cX$ has block decomposition $(B_1,B_2,\dots,B_\ell)$.

\item[(2)] Begin with $P_\lambda$ equal to the empty sequence, and for $B=B_1,B_2,\dots,B_\ell$ do the following:
\begin{enumerate}
\item[$\bullet$] If $B$ is of type (a), append to $P_\lambda$ the step $\begin{cases} \r,&\text{if $B=0$}; \\ \u_B,&\text{if $B\neq 0$}.\end{cases}$

\item[$\bullet$] If $B$ 
is $m\times m$ and of type (b), with $t_i = B_{i,i+1}$, append to $P_\lambda$ the sequence of steps 
\[\begin{cases} \(\r,\ \uu_{t_1,t_2},\ \uu_{t_3,t_4},\ \dots,\ \uu_{t_{m-2},t_{m-1}}\),&\text{if $m$ is odd}; \\
\(\d_{t_1},\ \uu_{t_2,t_3},\ \uu_{t_4,t_5},\ \dots,\ \uu_{t_{m-2},t_{m-1}}\),&\text{if $m$ is even}.\end{cases}\]  In particular, if $m=2$ then append the single step $\d_{t_{1}}$.
 
 \item[$\bullet$] If $B$ 
is $m\times m$ and of type (c), with $t_i = B_{i,i+1}$ and $u = B_{1,1}$, append to $P_\lambda$ the sequence of steps $ \(\u_u,\ \uu_{t_1,t_2},\ \uu_{t_3,t_4},\ \dots,\ \uu_{t_{m-2},t_{m-1}}\)$.
 \end{enumerate}
\end{enumerate}
For example, with $n=8$, the functional  $\lambda = r e_{1,3}^* + s e_{4,5}^* + t e_{4,6}^* + u e_{5,7}^* + v e_{7,8}^* \in \cX$ from (\ref{lambda-ex}) corresponds to 
\[(\d_r,\r,\u_s,\uu_{t,u},\u_v) \in \sLL(n,\FF_q),\] assuming $r,s,t,u,v \in \FF_q^\times$.
If at the beginning of an iteration of (2), the current endpoint of the path $P_\lambda$ lies on the line $\{x +y = k\}$, then at the end of the iteration, $P_\lambda$ will have advanced to a point on the line $\{x +y =k+m\}$ where $m$ is the size of the current block $B$.  As $P_\lambda$ begins at the origin and as the sum of the block sizes  is $n-1$,  the output path $P_\lambda$ indeed lies in $\sL(n,\FF_q)$.  Since by construction the first step in $P_\lambda$ is not $\uu$, in fact we have $P_\lambda \in \sLL(n,\FF_q)$.

Thus $\lambda \mapsto P_\lambda$ gives a well-defined map $\cX \to \sLL(n,\FF_q)$.  Checking that this map is the inverse of $P\mapsto \lambda_P$ is straightforward from the definitions.
To complete the proof of the theorem, we now must show that the characters $ \xi_\lambda$ for $\lambda \in \cX$ are the distinct Heisenberg characters of $\UT_n(\FF_q)$.
By (\ref{bell-thm}) and Lemma \ref{heis-obs}, an irreducible character of $\UT_n(\FF_q)$ is Heisenberg if and only if it appears as a constituent of a supercharacter $\chi_\Lambda$ indexed by a labeled set partition $\Lambda \in \sP(n,\FF_q)$ whose arcs are of the form $(i,i+1)$ or $(i,i+2)$.  
The set of such $\Lambda$ corresponds via (\ref{ide}) to precisely the set of functionals $\cY\subset \cX$ consisting of $\lambda \in \fkt_n(\FF_q)^*$ with the following property: the upper form of the matrix of $\lambda$ is block diagonal with blocks of types (a) or (b).

It suffices to show that if $\lambda \in \cY$ has block decomposition $(B_1,\dots,B_\ell)$,
then 
the distinct irreducible constituents of $\chi_\lambda$ are the characters $\xi_\mu$, where $\mu$ ranges over all functionals in $\cX$  whose block decompositions are obtained from the block decomposition of $\lambda$ by the following operation: place an arbitrary (possibly zero) element of $\FF_q$ in the upper left corner of each block $B_i$ of type (b) with odd dimension (yielding a block of type (c) if the element is nonzero).

If $\lambda=0$ then $\chi_\lambda=\xi_\lambda$ is irreducible and this statement holds trivially. Suppose the block decomposition of $\lambda$ has a single nonzero block.   If the nonzero block has type (a) then $\chi_\lambda = \xi_\lambda$ is irreducible and the desired statement again holds trivially, and if the block has type (b) then the desired statement 
is simply
Lemma \ref{tech-lem}.  

To treat the general case, write $\lambda = \lambda_1 + \lambda_2+\dots + \lambda_\ell$ where $\lambda_i \in \cY$ in the unique functional with block decomposition 
\be\label{decomp}(\underbrace{0,0,0,0,0\dots,0}_{m_1+\dots+m_{i-1}\text{ zeros}},\ B_i,\ \underbrace{0,0,0,0,0,\dots,0}_{m_{i+1}+\dots+m_\ell\text{ zeros}}),\qquad\text{where block $B_j$ is $m_j\times m_j$}.\ee  Write $\Irr(\chi)$ for the set of irreducible constituents of a character $\chi$ of $\UT_n(\FF_q)$.  It is enough to show two things: that  if we have $\mu_i \in \fkt_n(\FF_q)^*$ such that $\xi_{\mu_i} \in \Irr(\chi_{\lambda_i})$ for each $i \in [\ell]$, then $\xi_{\mu} \in \Irr(\chi_\lambda)$ for $\mu = \mu_1+\dots+\mu_\ell$; and that every irreducible constituent of $\chi_\lambda$ has this form.
To this end,
we note that well-known properties of the supercharacters of $\UT_n(\FF_q)$ (see \cite[Section 2.3]{Thiem}) imply that $\chi_\lambda = \prod_{i=1}^\ell \chi_{\lambda_i}$ and $\langle \chi_\lambda ,\chi_\lambda \rangle = \prod_{i=1}^\ell \langle \chi_{\lambda_i}, \chi_{\lambda_i} \rangle$.  The result \cite[Lemma 2.1]{M_C}  asserts that under these hypotheses the product map 
\be\label{prod-map} \barr{ccc} \prod_{i=1}^\ell \Irr(\chi_{\lambda_i}) & \to&  \Irr(\chi_{\lambda})\\
(\psi_1,\dots,\psi_\ell) & \mapsto & \prod_{i=1}^\ell \psi_i
\earr
\ee is a bijection.  Since the elements of $\Irr(\chi_{\lambda_i})$ and $\Irr(\chi_{\lambda})$ are all characters of the form $\xi_\mu$ by Proposition \ref{heis-cor},  if $\mu_i \in \fkt_n(\FF_q)^*$ such that $\xi_{\mu_i} \in \Irr(\chi_{\lambda_i})$ then $\prod_{i=1}^\ell \xi_{\mu_i} = \xi_{\mu}$ for some $\mu \in \fkt_n(\FF_q)^*$.  In light of the formula (\ref{xi-formula}) and the linear independence of the functions $\theta_\nu$ for $\nu \in \fkt_n(\FF_q)^*$, we deduce that we can take $\mu = \mu_1 + \dots +\mu_\ell$, which completes our argument.
\end{proof}

The degree of the Heisenberg character $\xi_P$ has a  simple formula in terms of $P \in \sLL(n,\FF_q)$.
\begin{corollary}\label{deg-cor} Let $n$ be any positive integer and $q>1$ any prime power.
\begin{enumerate}
\item[(1)] The Heisenberg character $\xi_P$ indexed by $P \in \sLL(n,\FF_q)$ has degree $\xi_P(1) = q^{k+\ell}$, where $k,\ell$ are the numbers of times the steps $\d, \uu$ respectively occur  in $P$.

\item[(2)] For integers $n\geq 2$ and $e\geq 0$,  the number of Heisenberg characters of $\UT_n(\FF_q)$ with degree $q^e$  is 
\[  q^{n-e-2}\(\binom{n-e-1}{e}(q-1)^e + \binom{n-e-2}{e}(q-1)^{e+1}\).\]
\end{enumerate}
\end{corollary}

\begin{remark}
This result with Corollary \ref{fib} proves Theorem \ref{thm2} in the introduction. 
The corollary also shows that  the number of Heisenberg characters of $\UT_n(\FF_q)$ with degree $q^e$ is a polynomial in $q-1$ with nonnegative integer coefficients.  
The same statement for the irreducible characters of $\UT_n(\FF_q)$ is a conjecture usually attributed to Lehrer \cite{Lehrer}; it has been verified in the cases when either $n\leq 13$ \cite{E} or $e \leq 8$ \cite{M_C}.  The number of Heisenberg characters of a pattern group can fail to be a polynomial in the size of the ambient field; see \cite[Theorem 4.9]{Hthesis}.
\end{remark}

\begin{proof}
Fix $P \in \sLL(n,\FF_q)$ and let $\lambda = \lambda_P \in \fkt_n(\FF_q)^*$.   Suppose $(B_1,\dots,B_s)$ is the block decomposition of $\lambda$, where block $i$ has size $m_i$.  Write $\lambda = \lambda_1+\dots+\lambda_s$ such that $\lambda_i \in \fkt_n(\FF_q)^*$ has block decomposition $(0,\dots,0,B_i,0,\dots,0)$ as in (\ref{decomp}), and let $P_i \in \sLL(n,\FF_q)$ be the unique path with $\lambda_i = \lambda_{P_i}$.  

For each $i$, the first $m_1+\dots + m_{i-1}$ steps and the last $m_{i+1}+\dots +m_s$ steps of the path $P_i$ are all the horizontal step $(1,0)$, and one can recover $P$ by removing these steps from each $P_i$ and concatenating the shortened paths which remain.  It follows that the parameters $k,\ell$ defined for $P$ are  simply the  sums of the respective parameters for the paths $P_1,\dots,P_s$.    Since we observed that $\xi_P = \prod_{i=1}^s \xi_{P_i}$ in the proof of Theorem \ref{heis-thm}, it suffices to prove part (1)  in the case when $\lambda$ has at most one nonzero block.  Therefore, without loss of generality assume $\lambda = \lambda_i$ for some $i$.

Recall the matrix types (a), (b),  (c) and the map $\lambda \mapsto P_\lambda$ 
defined in the proof of Theorem \ref{heis-thm}. If the unique nontrivial block $B_i$ is of type (a), then $\lambda$ is invariant under the two-sided action of $\UT_n(\FF_q)$ and $P$ uses only the steps $\r$ and $\u$, so $k=\ell=0$ and $\xi_P(1) = 1 = q^{k+\ell}$.  Suppose $B_i$ is of type (b).  If $m_i$ is odd then we see from the definition of the map $\lambda \mapsto P_\lambda$ that  $k=0$ and $\ell = \frac{m_i-1}{2}$, and from Lemma \ref{tech-lem1} that  the coadjoint orbit of $\lambda$ has size $q^{m_i-1}$, whence $\xi_P(1) = q^{k+\ell}$ by (\ref{xi-formula}).  If $m_i$ is even then similarly $k=1$ and $\ell = \frac{m_i-2}{2}$ and the coadjoint orbit of $\lambda$ has size $q^{m_i}$ by Lemma \ref{tech-lem}, so again $\xi_P(1) = q^{k+\ell}$.  Finally, if $B_i$ is of type (c) then $k=0$ and $\ell = \frac{m_i-1}{2}$ and by Lemma \ref{tech-lem1} the coadjoint orbit of $\lambda$ has size $q^{m_i-1}$, so $\xi_P(1) = q^{k+\ell}$ and we conclude that this holds for all $P \in \sLL(n,\FF_q)$.

To prove part (2),  observe that there are $\binom{e}{k}(q-1)^{e+k}$ distinct $\FF_q$-labeled paths from the origin to the point $(e-k,e+k)$ which use only the steps $\uu,\d$, and so there 
are  $(q-1)^eq^e$ distinct $\FF_q$-labeled paths from the origin to the line $x+y = 2e$ which use only the steps $\uu,\d$.  As there are $\binom{(n-2e-1)+e}{e}q^{n-2e-1}$ ways to insert $n-2e-1$ labeled steps of the form $\r,\u$ into such a path, we deduce that the number of paths in $\sL(n,\FF_q)$ with exactly $e$ steps of the form $\uu,\d$ is $\binom{n-e-1}{e} (q-1)^e q^{n-e-1}$, and that $\binom{n-e-2}{e-1} (q-1)^{e+1} q^{n-e-2}$ of these begin with $\uu$.  Subtracting the second of these quantities  from the first gives the desired formula after applying the usual recurrence for binomial coefficients.
\end{proof}

If $P \in \sLL(n,\FF_q)$ then the Heisenberg character $\xi_P$ of $\UT_n(\FF_q)$ is a constituent of the supercharacter indexed by $\lambda_P \in \fkt_n(\FF_q)^*$.  The character $\xi_P$ is equal to this supercharacter only if the latter is irreducible, and this occurs exactly what
the left and right $\UT_n(\FF_q)$-orbits of $\lambda_P \in \fkt_n(\FF_q)^*$ have a trivial intersection.  It is straightforward to see from our definition of $\lambda_P$ that this fails   if the 
path $P$  uses the step $\uu\ =(0,2)$.  Thus $\xi_P$ is a Heisenberg supercharacter only if $P \in \sK(n,\FF_q)\subset \sLL(n,\FF_q)$.  This necessary condition is in fact sufficient, since if $P \in \sK(n,\FF_q)$ then $\lambda_P$ is equal to the linear functional defined by 
the unique $\FF_q$-labeled set partition $\Lambda_P$ of $[n]$  satisfying the following condition:
\begin{enumerate}
\item[] $\Lambda_P$ has an arc $(i,j)$ labeled by $t \in \FF_q^\times$ if and only if $P$ has a step labeled by $t$ which travels from the point $(x,y-1)$ to the point $(x',y)$ such that $i=x+y$ and $j=x'+y+1$.
\end{enumerate}
It is routine to check that the correspondence $\sK(n,\FF_q) \to \NC(n,\FF_q)$ given by $P \mapsto \Lambda_P$ is well-defined, and that $\Lambda_P =\lambda_P$ as elements of $\fkt_n(\FF_q)^*$.  Thus, if we write $\chi_P$ for the supercharacter indexed by $\Lambda_P$ for $P \in \sK(n,\FF_q)$, then $\chi_P = \xi_P$ and the following theorem is immediate.

%
%

\begin{theorem} \label{del-thm} 
The map  
\[\barr{ccc} \sK(n,\FF_q) &\to
&
  \bigl\{ \text{Heisenberg supercharacters of $\UT_n(\FF_q)$}\bigr \}  \\
P & \mapsto &  \chi_{P}
\earr\]
is a bijection, and the number of Heisenberg supercharacters of $\UT_n(\FF_q)$
 is $\Del{n}{q-1}$.
\end{theorem}

\begin{remark}
Thus, as mentioned in the introduction, the Pell numbers  \cite[A000129]{OEIS} count the Heisenberg supercharacters of $\UT_n(\FF_2)$.
It is also not difficult to show that  the numbers of Heisenberg supercharacters of $\UT_n(\FF_3)$  form the sequence $\{ \Del{n}{2} \}_{n=0}^\infty = (0,1,3,11,39,139,495,\dots)$ \cite[A007482]{OEIS} whose $n$th term is the number of subsets of $[2n]$ in which each odd number has an even neighbor.
\end{remark}

\section{Counting invariant characters}
\label{c-heis-sect}

We now consider the Heisenberg characters of $\UT_n(\FF_q)$ which are invariant under the action by multiplication of the group's linear characters.  By Lemma \ref{heis-obs} there are $q^{n-1}$ linear characters of $\UT_n(\FF_q)$, and one easily confirms that they are precisely the functions 
$\theta_\tau : g \mapsto \theta\circ\tau(g-1)$ defined in Section \ref{superchar-defs} for $\tau \in \FF_q\spanning\{ e_{i,i+1}^* : 1\leq i < n\}$.  

\begin{proposition}\label{invar-prop} A Heisenberg character of $\UT_n(\FF_q)$ is invariant under multiplication by all linear characters  if and only if it is indexed by a labeled path in $\sLL(n,\FF_q)$ which uses only the steps $\d$ and $\uu$.  There are $q^{k-1}(q-1)^k$ such paths if $n=2k+1$ is odd and zero such paths if $n$ is even.
\end{proposition}

\begin{remark}
More generally, one can check by some tedious calculations that the size of the linear character orbit of the Heisenberg character indexed by $P \in \sLL(n,\FF_q)$ is $q^k$ where $k$ is the number of steps in $P$ given by $\r$ or $\u$.
\end{remark}

\begin{proof}
The Heisenberg character indexed by $P \in \sLL(n,\FF_q)$ is invariant under multiplication by all linear characters if and only if 
$\lambda_P$ and $\lambda_P + \tau$ belong to the same coadjoint orbit for all $\tau   \in \FF_q\spanning\{ e_{i,i+1}^* : 1\leq i < n\}$.  This follows  because $\xi_P$ is the Kirillov function indexed by $\lambda_P$, and so the formula (\ref{xi-formula}) (with the fact that $\tau$ is invariant under the two-sided action of the group) implies that the product of $\theta_\tau $ and $ \xi_P$ is the Kirillov function indexed by $\lambda_P + \tau$.

Because $\lambda_P$ is a linear combination of functionals of the form $e_{i,i+1}^*$ and $e_{i,i+2}^*$, it follows by inspection (if not from the preceding discussion) that every element of the coadjoint orbit of $\lambda_P$ has the form $\lambda_P+\tau$ for some $\tau   \in \FF_q\spanning\{ e_{i,i+1}^* : 1\leq i < n\}$.  We deduce from this that  
$\xi_P$ is invariant  if and only if the size of the coadjoint orbit of $\lambda_P$ is exactly $q^{n-1}$, or equivalently if $\xi_P$ has degree $q^{\frac{n-1}{2}}$.  The proposition is now an immediate consequence of 
 Corollary \ref{deg-cor}.
\end{proof}

For $n\geq 2$, we  define $\hom : \UT_n(\FF_q)\to \FF_q^+$ as the homomorphism $\hom(g) =\sum_{i=1}^{n-1} g_{i,i+1}$ and let $C$  denote the subgroup of linear characters $\psi$ of $\UT_n(\FF_q)$ which have $\psi(g)=1$ whenever $\hom(g) = 0$. 
One checks that this coincides with the set 
\[C = \left\{ \vartheta \circ \hom : \vartheta \in \Hom\(\FF_q^+, \CC^\times\) \right\}\]
given in the introduction.
Noting our description of the linear characters of the unitriangular group, we also note that
\[ C = \left\{ \theta_{t\gamma} : t \in \FF_q\right\},\qquad\text{where }\gamma = e_{1,2}^* + e_{2,3}^* + \dots + e_{n,n-1}^* \in \fkt_n(\FF_q)^*.\]
As a preliminary matter let us describe the $C$-invariant supercharacters of $\UT_n(\FF_q)$.  From the formula (\ref{superchar-def}),
it is clear that $\chi_\lambda$ is $C$-invariant if and only if the two-sided $\UT_n(\FF_q)$-orbit of $\lambda \in \fkt_n(\FF_q)^*$ contains $\lambda+t\gamma$ for all $t \in \FF_q$.  
Using this fact, the following lemma is essentially a straightforward technical exercise in linear algebra.

\begin{lemma}\label{fe-lem} The supercharacter of $\UT_n(\FF_q)$ indexed by $\Lambda \in 
 \sP(n,\FF_q)$ is $C$-invariant if and only if for each $j \in [n-1]$, there exists either some $(i,j+1) \in \Arc(\Lambda)$ with $i<j$ or some $(j,k+1) \in \Arc(\Lambda)$ with $j<k$.  
 \end{lemma}
 
\begin{remark} One could also prove this result by noting that the elements of $C$ are themselves supercharacters, and then applying either the formula \cite[Eq.\ 2.1]{Thiem} or \cite[Corollary 4.7]{Thiem}, which gives partial rules for decomposing  products of supercharacters as sums of supercharacters.
 \end{remark}

The \emph{matrix} of a labeled set partition $\Lambda \in \sP(n,\FF_q)$ is the $n\times n$ matrix $\sum_{(i,j) \in \Arc(\Lambda)} \Lambda_{ij} e_{ij}$; associating a set partition to its matrix gives a bijection from $\sP(n,\FF_q)$ to the set of strictly upper triangular $n\times n$ matrices over $\FF_q$ with at most one nonzero entry in each row and column.  Shifting the matrix of a set partition one column to the right correspond to the injective map 
\[\label{shift}\shift : \sP(n,\FF_q) \to \sP(n+1,\FF_q)\] given by defining $\shift(\Lambda)$ 
to be the $\FF_q$-labeled set partition of $[n+1]$ with arc set  $\{ (i,j+1) : (i,j) \in \Arc(\Lambda)\}$ and labeling map $(i,j+1) \mapsto \Lambda_{ij}$.
 If we ignore all labels, then $\shift$ is the left inverse of the ``reduction algorithm'' presented in \cite{reduction1}.

We say that a set partition is \emph{feasible} if each of its blocks has at least two elements.
We have appropriated this term from another context:
  in \cite{Bernhart}, Bernhart alternately refers  to partitions of $[n]$ whose blocks do not contain any of the sets $\{1,2\}, \{2,3\}, \dots, \{n-1,n\}, \{n,1\}$ as subsets   as ``cyclically spaced'' or ``feasible.''  Such partitions are in bijection with  partitions of $[n]$ whose blocks have at least two elements 
   (see \cite[Section 3.5]{Bernhart}), and so we decide to call  partitions  of the latter type ``feasible'' since these objects seem to lack a well-known name in the literature.
  
Denote the number of feasible partition of $[n]$ with $k$ blocks by $\stirlstirl{n}{k}$.  These numbers are sometimes called the \emph{associated Stirling numbers of the second kind} \cite{Comtet} and are listed as sequence  \cite[A008299]{OEIS}.
The number of feasible $\FF_q$-labeled set partition of $[n]$ is then given by $\Fe{n}{q-1}$ where 
\[\Fe{n}{x} \omdef= \sum_{k=0}^n \stirlstirl{n}{k} x^{n-k}.\]
The integer sequence $\{\Fe{n}{1}\}_{n=0}^\infty=(1, 0, 1, 1, 4, 11, 41, \dots)$ counting the
 feasible  partitions of $[n]$ appears as \cite[A000296]{OEIS}.  
 
Since a feasible set partition $\Lambda\vdash[n]$ has some $(i,j) \in \Arc(\Lambda)$ or some $(j,k) \in \Arc(\Lambda)$ for each $j \in [n-1]$, the following result is clear from Lemma \ref{fe-lem}.

\begin{proposition}\label{fe-thm} The correspondence $\Lambda \mapsto \chi_{\shift(\Lambda)}$ is a bijection 
\[  \{\text{Feasible elements of }\sP(n,\FF_q)\} \to
  \{ \text{$C$-invariant supercharacters of $\UT_{n+1}(\FF_q)$} \}.  \] Hence $\UT_{n+1}(\FF_q)$ has  $\Fe{n}{q-1}$ distinct $C$-invariant supercharacters.
    \end{proposition}

Recall the definitions of $\sP(n,\FF_q)$ and $\NC(n,\FF_q)$ from the introduction.
The cardinalities of these sets are given respectively by $\Bell{n}{q-1}$ and $\Cat{n}{q-1}$, where
\[\Bell{n}{x} \omdef =
\sum_{k=0}^n \stirl{n}{k} x^{n-k}\quad\text{and}\quad 
\Cat{n}{x} \omdef =
 \sum_{k=0}^n N(n,k) x^{n-k}.\]  Here $\stirl{n}{k}$ and $N(n,k)$ are the {Stirling numbers of the second kind} and the {Narayana numbers}, defined as the number of ordinary and noncrossing set partitions of $[n]$ with $k$ blocks (equivalently, with $n-k$ arcs). We note the well-known formula $N(n,k) = \frac{1}{n} \binom{n}{k} \binom{n}{k-1}$ for $n> 0$ and adopt the convention $\stirl{0}{k} = N(0,k) = \delta_{k}$. 
 In explanation of our notation, $\Bell{n}{1}$ is the $n$th {Bell number}  and $\Cat{n}{x}$ is  the \emph{Narayana polynomial}, whose values give the Catalan numbers    when $x=1$ and the 
little Schr\"oder numbers 
when $x=2$.

Let $\cC_n = \Cat{n}{1} = \frac{1}{n+1}\binom{2n}{n}$ denote the $n$th Catalan number.  
The polynomials $\Bell{n}{x}$ and $\Cat{n}{x}$  have the noteworthy alternate formulas 
\be\label{alt-eq}\Bell{n+1}{x} =
 \sum_{k=0}^n \binom{n}{k}\Fe{k}{x} (x+1)^{n-k}
 \quad\text{and}\quad
 \Cat{n+1}{x} =\sum_{k=0}^{\lfloor n/2\rfloor} \cC_k \binom{n}{2k} x^k (x+1)^{n-2k}\ee
for all $n\geq 0$ \cite[Theorem 3.2]{M_new}.  The right hand side in particular is a well-known identity due to Coker \cite{Coker_Enum}, which will be of use in the following corollary.

\begin{corollary} \label{c-irr-thm}  Fix a nonnegative integer $n$ and a prime power $q>1$.
\begin{enumerate}
\item[(1)] The number of $C$-invariant irreducible supercharacters of $\UT_{n+1}(\FF_q)$ is 
\[(1-q)^{\lfloor n/2\rfloor} \Cat{n+1}{-1} = \begin{cases} (q-1)^{n/2} \cC_{n/2},&\text{if $n$ is even}; \\ 0,&\text{if $n$ is odd}.\end{cases}\]

\item[(2)] The number of $C$-invariant Heisenberg supercharacters of $\UT_{n+1}(\FF_q)$ is 
\[(1-q)^{\lfloor n/2\rfloor} \Del{n+1}{-1} = \begin{cases} (q-1)^{n/2},&\text{if $n$ is even}; \\ 0,&\text{if $n$ is even}.\end{cases}\]
\end{enumerate}
\end{corollary}

\begin{proof}
The equality of the expressions in (1) and (2) follows from (\ref{alt-eq}) and Proposition \ref{preHe-formulas}.
As discussed in \cite{reduction1,M_new}, if $\Lambda \in \sP(n-1,\FF_q)$ then $\shift(\Lambda)$ is noncrossing if and only if $\Lambda$ is noncrossing and has no blocks with more than two elements.  It follows from the preceding proposition that the $C$-invariant irreducible supercharacters of $\UT_{n+1}(\FF_q)$ are in bijection with the $\FF_q$-labeled noncrossing partitions of $[n]$ whose blocks all have size two. 
Such partitions clearly exist only if $n$ is even, in which case their number is equal to the given expression because
the number of (unlabeled) noncrossing partitions of $[2k]$ whose blocks all have size two is $\cC_{k}$ by \cite[Exercise 6.19o]{Stanley}.

For the second part, observe by Lemma \ref{heis-obs} that the supercharacter $\chi_{\shift(\Lambda)}$ is Heisenberg 
if
and
only
if
$\shift(\Lambda)$ is noncrossing and has no arcs $(i,j)$ with $j>i+2$.  Consulting the definition of $\shift$, we see that this is equivalent to the condition that  $\Lambda$ be noncrossing with no blocks with more than two elements and no arcs $(i,j)$ with $j>i+1$.  It follows that $\chi_{\shift(\Lambda)}$ is Heisenberg and $C$-invariant  if and only if the unlabeled form of $\Lambda $ is $\{ \{1,2\}, \{3,4\}, \dots, \{2k-1,2k\}\}$ for some $k$.  
Hence there are $(q-1)^k$ distinct $C$-invariant Heisenberg supercharacters of $\UT_{n+1}(\FF_q)$ when $n=2k$ and no such supercharacters if $n$ is odd.
\end{proof}

Let us now return our attention to Heisenberg characters.
We do not know of a combinatorial bijection from the set of $C$-invariant Heisenberg characters of $\UT_{n+1}(\FF_q)$ to the family of labeled lattice paths $\sMM(n,\FF_q)$, but we can prove that these two sets are equinumerous.

\begin{theorem}\label{c-heis-thm}
The number of $C$-invariant Heisenberg characters of $\UT_{n+1}(\FF_q)$
 is  $\In{n}{q-1}$.
  \end{theorem}

To show this, we first derive an explicit formula for the number of Heisenberg characters of $\UT_{n+1}(\FF_q)$.
 In the following lemma, by a composition of a number $n$, we mean a sequence of positive integers whose sum is $n$.  The positive integers are the \emph{parts} of the composition. 


\begin{lemma}\label{compos}
For each composition $\textbf{c} = (c_1,\dots,c_\ell)$ of a positive integer $n$, define 
\[f_q(\textbf{c}) = \prod_{i=1}^\ell \((q-1)^{c_i-1} - \sin(c_i \pi /2) (q-1)^{(c_i-1)/2}\).
\]
The number of $C$-invariant Heisenberg characters of $\UT_{n+1}(\FF_q)$ is then the sum
$\sum_{\textbf{c}} f_q(\textbf{c})$ over all compositions $\textbf{c}$ of $n$.
\end{lemma}


\begin{proof}
The sum in the lemma is 0 if $n=1$, which is the number of $C$-invariant Heisenberg characters of $\UT_{2}(\FF_q)$ since in that case $C$ is   equal to the set of all irreducible characters of $\UT_2(\FF_q)\cong \FF_q^+$.

Assume $n\geq 2$ and let $G = \UT_{n+1}(\FF_q)$ and $\fkn = \fkt_{n+1}(\FF_q)$.
Choose $\Lambda \in \sP(n+1,\FF_q)$, and view this labeled set partition as a functional in $\fkn^*$ in the usual way via (\ref{ide}).   By Lemma \ref{heis-obs}, if $\psi$ is a $C$-invariant Heisenberg constituent of $\chi_\Lambda$, then $\chi_\Lambda$ is $C$-invariant and every $(i,j) \in \Arc(\Lambda)$ has $j<i+3$ since $\ker \Lambda \supset \fkn^3$.  
By Lemma \ref{fe-lem}, it follows that every $C$-invariant Heisenberg character of $G$ is a constituent of a supercharacter $\chi_\Lambda$ for some 
$\Lambda \in \sP(n+1,\FF_q)$ which has 
$\Arc(\Lambda) \subset \{ (i,i+2) : i \in [n-1]\}$
such that if $(i,i+2) \notin \Arc(\Lambda)$ for some $i \in [n-1]$, then both $(i-1,i+1)$ and $(i+1,i+3)$ belong to $ \Arc(\Lambda)$.  

Recall the definition of the block decomposition of an element of $\fkn^*$ given just after Theorem \ref{heis-thm}, and define $\cX$ as the set of $\lambda \in \fkn^*$ with the property that 
each matrix $B$ in the block decomposition of $\lambda$ is $m\times m$ with $m\geq 2$ and has $B_{ij} \neq 0$ if and only if $j=i+1$.
It is not difficult to see that our previous
assertion is equivalent to the statement that every $C$-invariant Heisenberg character of $G$ is a constituent of a supercharacter $\chi_\lambda$ for a unique $\lambda \in \cX$.

Fix $\lambda \in \cX$ and let $\gamma = \sum_{j=1}^{n} e_{j,j+1}^* \in \fkn^*$.  
  The constituents of  $\chi_\lambda$ are the Heisenberg characters $\xi_\mu$ for $\mu \in G\lambda G$ and, as noted above, the elements of $C$ are the functions $ \theta_{t\gamma}$ for $t \in \FF_q$.  As in the proof of Proposition \ref{invar-prop}, since $\ker \gamma \supset \fkn^2$, the formula (\ref{xi-formula}) implies 
  $\theta_{t\gamma}\xi_\mu = \xi_{\mu + t\gamma}$ and so  
  the character $\xi_\mu$ for $\mu \in G\lambda G$ is $C$-invariant if and only if $\mu + \FF_q\spanning\{\gamma\}$ is a subset of the coadjoint orbit of $\mu$.   By construction, the two-sided $G$-orbit of $\lambda$ is $G\lambda G = \lambda + \FF_q\spanning \{ e_{j,j+1}^* : j \in [n] \}$ and so $\mu -\lambda$   is invariant under the two-sided action of $G$ for each $\mu \in G\lambda G$.  Consequently, if $\mu \in G\lambda G$ then $\mu + \FF_q\spanning\{\gamma\}$ is a subset of the coadjoint orbit of $\mu$ if and only if $\lambda + \FF_q\spanning\{\gamma\}$ is a subset of the coadjoint orbit of $\lambda$.  Thus,   
  if $\lambda +\FF_q\spanning\{\gamma\}$ is a subset of the coadjoint orbit of $\lambda$ then every irreducible constituent of $\chi_\lambda$ is a $C$-invariant Heisenberg character, while no irreducible constituents of $\chi_\lambda$ are $C$-invariant otherwise.

Let $(B_1,\dots,B_\ell)$ be the block decomposition of $\lambda$, and let $m_i\geq 2$ be the dimension of block $B_i$.  Note that each $B_i$ has a nonzero entry in each superdiagonal position, and zeros in all of other positions.
The sequence of positive numbers $(m_1,\dots,m_\ell)$ is a composition of $n$; call this the \emph{shape} of $\lambda \in \cX$.
As in the proof of Theorem \ref{heis-thm}, write $\lambda = \lambda_1 +\dots + \lambda_\ell$ where each $\lambda_i \in \fkn^*$ is the unique functional with block decomposition $(0,\dots,0,B_i,0,\dots,0)$ as in (\ref{decomp}).  Fix $i \in [\ell]$; then
$\lambda_i$ has the form
$\lambda_i=
t_j e_{j,j+2}^* + \dots + t_k e_{k,k+2}^*$ for some integers $1\leq j \leq k \leq n-1$ with $m_i = k-j+2$ where $t_j,\dots,t_k \in \FF_q^\times$, and we define  
$\gamma_i\in  \fkn^*$ by $\gamma_i = e_{j,j+1}^* + \dots + e_{k,k+1}^* + e_{k+1,k+2}^*$.
Then $\gamma = \gamma_1+\dots+\gamma_\ell$, and 
 a straightforward computation shows that $\lambda + \FF_q\spanning\{\gamma\}$ is a subset of the coadjoint orbit of $\lambda$ if and only if $\lambda_i + \FF_q\spanning\{\gamma_i\}$ is a subset of 
the coadjoint orbit of each $\lambda_i$.
By Lemmas \ref{tech-lem1} and \ref{tech-lem}, the coadjoint orbit of $\lambda_i$ is formed by adding $\lambda_i$ to all elements in some subspace of $\fkn^*$.  Hence,  every irreducible constituent of $\chi_\lambda$ is $C$-invariant if $\lambda_i + \gamma_i$ belongs to the coadjoint orbit of $\lambda_i$ for each index $i$, and $\chi_\lambda$ has no $C$-invariant irreducible constituents otherwise.  In the latter case,  the $C$-orbits of the irreducible constituents of $\chi_\lambda$ all have order $q$,
because $C$ acts freely on these constituents  since $\chi_\lambda$ is $C$-invariant.

Each $\lambda_i$ is determined by a choice of $m_i-1$ entries  $t_j,\dots,t_k \in \FF_q^\times$.  
If $m_i$ is even, then  part (1) of Lemma \ref{tech-lem} asserts that $\lambda_i + \gamma_i$ belongs to the coadjoint orbit of $\lambda_i$ for all $(q-1)^{m_i-1}$ choices of these entries.  If $m_i$ is odd, then part (2) of Lemma \ref{tech-lem} gives a formula the number of choices of entries which result in $\lambda_i+\gamma_i$ belonging to the coadjoint orbit of $\lambda_i$.  
From these lemmas, we conclude that
the number of functionals $\lambda \in \cX$ with a fixed shape $(m_1,\dots,m_\ell)$, for which all irreducible constituents of $\chi_\lambda$ are $C$-invariant, is the product $\prod_{i=1}^\ell \wt f_q(m_i)$, where 
\[ \wt f_q(m) = \begin{cases} (q-1)^{m-1},&\text{if $m$ is even;} \\
\frac{1}{q}\( (q-1)^{m-1} - (-1)^{(m-1)/2}(q-1)^{(m-1)/2}\),&\text{if $m$ is odd.}
\end{cases}
\]
By Lemma \ref{tech-lem}, the number of irreducible constituents of the supercharacter $\chi_{\lambda_i}$ is one if $m_i$ is even and $q$ if $m_i$ is odd.  Thus,  by \cite[Lemma 2.1]{M_C} (see Eq. (\ref{prod-map}) above), the number of irreducible constituents of $\chi_\lambda$ depends only on the shape $\textbf{m} = (m_1,\dots,m_\ell)$ and is $q^{|\{ i\in [\ell] : m_i \text{ is odd}\}|}$.  The product of this quantity with $\prod_{i=1}^\ell \wt f_q(m_i)$ is exactly $f_q(\textbf{m})$, and we conclude that $f_q(\textbf{m})$ is the number of $C$-invariant Heisenberg characters which appear as constituents of supercharacters indexed by functionals in $\cX$ with the shape $\textbf{m}$.  The sum of $f_q(\textbf{m})$ over all shapes $\textbf{m}$ of elements of $\cX$ is thus equal to the number of $C$-invariant Heisenberg characters of $\UT_{n+1}(\FF_q)$.  These shapes range over all compositions of $n$ whose parts are all at least two.  Since $f_q(\textbf{m}) =0$ if $\textbf{m}$ has any parts less than two, the proof  is complete. \end{proof}

The following observation is implicit in the arguments in  this proof.

\begin{observation}\label{heis-orb-obs}
The $C$-orbit of a Heisenberg character or supercharacter of $\UT_n(\FF_q)$ has size 1 or $q$.
\end{observation}

\begin{proof} It is an easy exercise to deduce from the supercharacter formula \cite[Eq.\ 2.1]{Thiem} that 
 the $C$-orbit of a supercharacter has size 1 or $q$.
We noted at the end of the fifth paragraph of the preceding proof that if a Heisenberg character is a constituent a $C$-invariant supercharacter, then its $C$-orbit has size 1 or $q$.   If a Heisenberg character is a constituent of a supercharacter which is not $C$-invariant, then the $C$-orbit of the Heisenberg character has size $q$ since each character in the orbit is a constituent of a distinct supercharacter.
\end{proof}

Returning to the proof of Theorem \ref{c-heis-thm}, we note that one can deduce from the generating function of $\preIn{n}{x}$ that
\be\label{hec-rec} 
\In{n}{x}= \begin{cases}
0, & \text{if $n=1$};\\
x, & \text{if $n=2$};\\
x(x+1), & \text{if $n=3$};\\
x \In{n-1}{x} + x(x+1)\In{n-3}{x},&\text{if $n\geq 4$}.\end{cases}\ee
We  now demonstrate that the formula in Lemma \ref{compos} also satisfies this recurrence.

\begin{proof}[Proof of Theorem \ref{c-heis-thm}]
Fix $q$ and write $a_n$ for the number of $C$-invariant Heisenberg characters of $\UT_{n+1}(\FF_q)$.  We compute from  Lemma \ref{compos} that $a_n =\In{n}{q-1}$ for $n=1,2,3$, so to show $a_n =\In{n}{q-1}$ for all $n\geq 1$, it suffices to prove  $a_n = (q-1)a_{n-1} + q(q-1) a_{n-3}$ for $n\geq 4$. 
To this end, observe that 
the formula in  Lemma \ref{compos} implies
\be\label{a_n-id} a_n = \sum_{k=2}^{n} \((q-1)^{k-1} - \sin(k\pi /2) (q-1)^{(k-1)/2}\) a_{n-k},\qquad\text{for }n\geq 2.\ee
In particular, each term in this sum corresponds to the sum $\sum_{\textbf{c}} f_q(\textbf{c})$ over all 
compositions $\textbf{c}$ of $n$ whose parts are all at least two and whose last part is $k$.
Fix $n\geq 4$.  Applying (\ref{a_n-id}) to $a_n$ and $a_{n-1}$, we obtain
\[
a_n -(q-1)a_{n-1} = (q-1) a_{n-2}
+
\sum_{k=2}^{n-1} \(\sin(k\pi/2)(q-1)^{(k+1)/2} - \cos(k\pi /2) (q-1)^{k/2}\) a_{n-1-k}.
\]
Expanding the $a_{n-2}$ term on the right hand side via a second application of (\ref{a_n-id}), this equation becomes
\[ a_n -(q-1)a_{n-1} = 
(q-1)a_{n-3}
+
\sum_{k=4}^{n-1} \((q-1)^{k-1}+\sin(k\pi/2)(q-1)^{(k+1)/2}\) a_{n-1-k}
.\]  Shifting indices and comparing with (\ref{a_n-id}) shows that the sum on the right hand side is $(q-1)^2 a_{n-3}$; notably, when $n=4$ and the sum is zero, this still holds because $a_1 = 0$.  Thus, $a_n = (q-1)a_{n-1} + q(q-1) a_{n-3}$
for $n\geq 4$ as required.
\end{proof}

\section{Applications to the alternating subgroup of $\UT_n(\FF_q)$}\label{alt}
\label{alt-sect}

Fix $n\geq 2$ and recall that $\hom : \UT_n(\FF_q) \to \FF_q^+$ is the homomorphism $g \mapsto \sum_{i=1}^{n-1} g_{i,i+1}$.  When $q=2$, composing this map with the nontrivial homomorphism $\theta : \FF_2 \to \CC^\times$ yields the unique homomorphism
$\UT_n(\FF_2)\to \CC^\times$ which assigns $-1$ to each of the $n-1$ generators $1+e_{i,i+1}$.  In this sense, $\hom$ serves as something of a unitriangular analogue for the sign homomorphism of the symmetric group, and we therefore refer to the kernel of $\hom$, i.e., the algebra group 
\[ \UT^\hom_n(\FF_q) \omdef = \left\{ g \in \UT_n(\FF_q) : \sum_{i=1}^{n-1}g_{i,i+1} = 0\right\}, \qquad\text{for }n\geq 2,\]
as the \emph{alternating subgroup} of $\UT_n(\FF_q)$.  
We discussed this algebra group briefly in Example 5.1 in \cite{M_normal}, which computed a formula for the number of its supercharacters.  Here, using the results of the previous section, we give a more comprehensive set of formulas, counting in addition the numbers of irreducible supercharacters, Heisenberg supercharacters, and Heisenberg characters. 

To derive these formulas we require the following lemma, a straightforward exercise in the Clifford theory.

\begin{lemma}\label{alt-lem}
Fix an integer $n\geq 2$ and a prime power $q>1$. The following statements hold for supercharacters:

\begin{enumerate}
\item[(a)] Each supercharacter of $\UT_n(\FF_q)$ restricts to a supercharacter of $\UT^\hom_n(\FF_q)$, and every supercharacter  of $\UT^\hom_n(\FF_q)$ arises in this way. 
\item[(b)] Two supercharacters of $\UT_n(\FF_q)$ have the same restriction to $\UT^\hom_n(\FF_q)$ if and only if they belong to the same $C$-orbit.
\item[(c)]
A supercharacter $\chi$ of $\UT_n(\FF_q)$ restricts to an irreducible character of $\UT^\hom_n(\FF_q)$ if and only if $\chi$ is irreducible and not $C$-invariant.
\end{enumerate}
Likewise, the following statements hold for Heisenberg characters:
\begin{enumerate}
\item[(d)] A Heisenberg character $\psi$ of $\UT_n(\FF_q)$ restricts to a sum of $q$ distinct Heisenberg characters of $\UT^\hom_n(\FF_q)$ if $\psi$ is  $C$-invariant, and to a single Heisenberg character of $\UT^\hom_n(\FF_q)$ otherwise. 

\item[(e)] Every Heisenberg character of $\UT^\hom_n(\FF_q)$ appears as an irreducible constituent of the restriction of a Heisenberg character of $\UT_n(\FF_q)$.  

\item[(f)] The restrictions to $\UT^\hom_n(\FF_q)$ of two Heisenberg characters $\psi,\psi'$ of $\UT_n(\FF_q)$ are equal if $\psi$ and $\psi'$ belong to the same $C$-orbit, and they share no irreducible constituents otherwise.
\end{enumerate}

\end{lemma}

\begin{proof}
Let $G = \UT_n(\FF_q)$ and $H = \UT^\hom_n(\FF_q)$, and note that $|C| = |G|/|H|=q$ since $n\geq 2$.  
Problem (6.2) in \cite{Isaacs} asserts that if $\psi$ is an irreducible character of $H$ then $\Ind_H^G(\psi) = f \sum_{i=1}^s \chi_i$ where $f$ is an integer and the $\chi_i$ comprise a $C$-orbit of irreducible characters of $G$.  This fact with Frobenius reciprocity suffices to prove (f).
For any $i$, we have $q\psi(1) = fs \chi_i(1)$ and $f = \langle \psi, \Res_H^G(\chi_i) \rangle_H$ by Frobenius reciprocity.  Thus $\psi = \Res_H^G(\chi_i)$ if and only if $f=1$ and $s=q$, and we conclude that an irreducible character $\chi$ of $G$ restricts to an irreducible character of $H$ if and only if the $C$-orbit of $\chi$ has size $q$.

Example 5.1 in \cite{M_normal} derives (a) and (b)  from \cite[Lemma 5.1]{M_normal}.  Certainly only an irreducible supercharacter of $G$ can restrict to an irreducible supercharacter $\chi$ of $H$, and it follows from the previous paragraph and Corollary \ref{heis-orb-obs} that if $\chi$ is irreducible then this occurs if and only if  $\chi$ is not $C$-invariant. This proves (c).
%

Part (e) is immediate from the fact that $[G,[G,G]] = [H,[H,H]] = 1 + (\fkt_n(\FF_q))^3$, the proof of which we leave as an exercise.  
To prove (d),
let $\chi$ be a Heisenberg character of $G$.  By 
Clifford's theorem \cite[Theorem (6.2)]{Isaacs}, the restriction of $\chi$ to $H$ has the form $\Res_H^G(\chi) = e \sum_{i=1}^t \psi_i$ for a positive integer $e$ and $t$ distinct Heisenberg characters $\psi_i$ of $H$ with the same degree.  
If $\chi$ is $C$-invariant, then $\chi(g) = 0$ for all $g \in G-H$ so 
$ 1 = \langle \chi,\chi\rangle_G = \tfrac{|H|}{|G|} \langle \Res_H^G(\chi) ,\Res_H^G(\chi) \rangle_H =  e^2t/q$ and $e^2 t = q$.  On the other hand, $\chi(1) = e t \psi_1(1)$ and both $\chi(1)$ and $\psi_1(1)$ are powers of $q$ by \cite[Theorem A]{I95}.  Hence $e$ and $t$ are both powers of $q$; since these numbers are integers with $e^2t=q$, we must have $e=1$ and $t=q$. Thus a $C$-invariant Heisenberg character of $G$ restricts to a sum of $q$ distinct Heisenberg characters of $H$.
If $\chi$ is not $C$-invariant then by Observation \ref{heis-orb-obs}
the $C$-orbit of $\chi$ has order $q$, and it follows as in the supercharacter case that $\chi$ restricts to an irreducible character of $H$.  
%
\end{proof}

We are now prepared to list the character counting formulas promised for $\UT^\hom_n(\FF_q)$. 
For each positive integer $n$, define the following rational functions in an indeterminate $x$:
\begin{enumerate}
\item[] $\AltBell{n+1}{x}  \omdef=\tfrac{1}{x+1} \Bigl( \Bell{n+1}{x} -\Fe{n}{x}\Bigr) + \Fe{n}{x}$;
\item[] $\AltCat{n+1}{x} \omdef=\tfrac{1}{x+1} \Bigl(\Cat{n+1}{x} - (-x)^{\lfloor n/2\rfloor } \Cat{n+1}{-1}\Bigr)$; 
\item[] $\AltDel{n+1}{x} \omdef=\tfrac{1}{x+1} \Bigl( \Del{n+1}{x} - (-x)^{\lfloor n/2\rfloor} \Del{n+1}{-1}\Bigr)$; 
\item[] $\ds\AltHe{n+1}{x}\omdef = \tfrac{1}{x+1}\(\He{n+1}{x} -\In{n}{x}\) + (x+1)\In{n}{x}$.
\end{enumerate}
It is immediate from Lemma \ref{alt-lem} combined with Proposition \ref{fe-thm}, Corollary \ref{c-irr-thm}, and Theorem \ref{c-heis-thm}  that $\AltBell{n}{q-1}$, $\AltCat{n}{q-1}$, $\AltDel{n}{q-1}$, and $\AltHe{n}{q-1}$ respectively count the supercharacters, irreducible supercharacters, Heisenberg supercharacters, and Heisenberg characters of the alternating subgroup $\UT^\hom_n(\FF_q)$.
To see that these rational functions are actually polynomials in $x$ with nonnegative integer coefficients, we first note from Eq. (\ref{alt-eq}) and Proposition \ref{preHe-formulas} that 
\begin{enumerate}
\item[] $\ds\AltBell{n+1}{x}  \omdef=  \sum_{k=0}^{n} \binom{n}{k} \Fe{n-k}{x} (x+1)^{\delta_{k}+k-1}$ ,
\item[] $\ds\AltCat{n+1}{r} \omdef=\sum_{k=0}^{\lfloor (n-1)/2\rfloor} \cC_k \binom{n}{2k} x^k (x+1)^{n-2k-1}$,
\item[] $\ds\AltDel{n+1}{r} \omdef=\sum_{k=0}^{\lfloor (n-1)/2\rfloor}  \binom{n-k}{k} x^k (x+1)^{n-2k-1}$.
\end{enumerate}
The function $\AltHe{n+1}{x}$ is likewise a polynomial with nonnegative coefficients since by Corollary \ref{He-formulas}, 
$ \AltHe{2}{x} = 1$ and $\AltHe{3}{x} = (x+1)^2$ 
and  for $n\geq 4$, both $\He{n+1}{x}$ and $\In{n}{x}$ are divisible by $x+1$.  
%
These observation provide us with the following result, which implies Theorem \ref{thm3} in the introduction.

\begin{theorem} \label{alt-thm}
The polynomials $\AltBell{n}{x}$, $\AltCat{n}{x}$, $\AltDel{n}{x}$, and $\AltHe{n}{x}$ have nonnegative integer coefficients, and their values at $x=q-1$ are respectively the numbers of supercharacters, irreducible supercharacters, Heisenberg supercharacters, and Heisenberg characters of the alternating subgroup $\UT^\hom_n(\FF_q) \subset \UT_n(\FF_q)$ for $n\geq 2$.
\end{theorem}

%

\begin{remarks}
Some remarks are in order.
\begin{enumerate}
\item[(i)] Proposition 5.2 in \cite{M_normal} gives the alternate formula 
\[ \AltBell{n+1}{x} = \sum_{k=0}^n \frac{x^k + (-1)^k x}{x+1} \binom{n}{k} \Bell{n-k}{x},\qquad\text{for }n\geq 1.\] 
Noting the well-known recurrence $\Bell{n+1}{x} = \sum_k x^k \binom{n}{k}  \Bell{n-k}{x} $ \cite{KerberFollow,Rogers,Yan}, one finds that 
the
sequences of differences $\{ \AltBell{n}{1} - \Bell{n-1}{1}\}_{n=2}^\infty=(0,1,3,13,55,256,1274,\dots)$ \cite[A102287]{OEIS} and $\{ \Bell{n}{1}-\AltBell{n}{1}\}_{n=2}^\infty = (1,2,7,24,96,418,1989,\dots)$ \cite[A102286]{OEIS} give respectively the  numbers of distinguished even and odd blocks in all set partitions of $[n-1]$.  By \emph{distinguished}, we mean that we consider two blocks to be equal if and only if they are equal as sets and belong to the same set partition.  
For example, the set partitions of $\{1,2,3\}$ (in abbreviated notation) are $1|2|3$, $12|3$, $13|2$, $1|23$, and $123$, and they contain 10 distinguished blocks, of which 3 are even and 7 are odd.  

\item[(ii)] If $C(z)$ is the ordinary generating function for $\cC_n=\Cat{n}{1}$ and $D(z) = 1-2z+zC(z^2)$ then one finds that $\frac{1}{2}\( C(z) -D(z)\) = \sum_{n\geq 2} \AltCat{n}{1} z^n$.  Comparing this to the generating function listed for sequence A000150 in \cite{OEIS},  we find that
 $\{ \AltCat{n}{1} \}_{n=2}^\infty = (1,2,7,20,66,212,715,\dots)$ gives
the number of Dyck paths with $2n$ steps having an odd number of peaks at even height.  These numbers also count the number of dissections of an $n$-gon, rooted at an exterior edge, which are asymmetric with respect to that edge.  

\item[(iii)] The integers $\{ \AltDel{n}{1} \}_{n=0}^\infty = (0,1,2,6,14,35,84,\dots)$ appear as sequence A105635 in \cite{OEIS}, which identifies these numbers as the row sums of two related integer  arrays.


\item[(iv)] A polynomial $f(x)$ of degree $n$ is \emph{palindromic} if $f(x) = x^n f(x^{-1})$, so that its coefficients are the same read forwards and backwards.   It follows from  (\ref{alt-eq}), Proposition \ref{preHe-formulas}, and the formulas above that 
the polynomials $\Cat{n}{x}$, $\Del{n}{x}$, $\AltCat{n}{x}$, and $\AltDel{n}{x}$ are all palindromic.

\end{enumerate}
\end{remarks}

\end{document}